\documentclass[10pt,a4paper]{article}
\usepackage{typearea}
\usepackage{amsfonts}
\usepackage{amssymb}
\usepackage{amsthm}
\usepackage{amscd}
\usepackage[centertags]{amsmath}
\usepackage{eucal}
\usepackage{epsfig}
\usepackage[matrix,arrow,curve]{xy}
\CompileMatrices
\typearea{11}
\author{Andreas Klein}
\title{Symplectic monodromy, quasi-homogeneous polynomials and spectral flow}

\renewcommand{\ker}{\mathrm{ker}}

\newtheorem{theorem}{Theorem}[section]
\newtheorem{Def}[theorem]{Definition}
\newtheorem{prop}[theorem]{Proposition}
\newtheorem{lemma}[theorem]{Lemma}
\newtheorem{folg}[theorem]{Corollary}

\begin{document}
\maketitle
\begin{abstract}
We encode the variation structure of a quasihomogeneous polynomial with an isolated singularity as introduced by Nemethi in a set of spectral flows of the signature operator on the Milnor bundle by varying global elliptic boundary conditions in a specific way using the quasihomogeneous circle action on the Brieskorn lattice. For this, we use adiabatic techniques and well-known results on spectral flow and Maslov index. Furthermore we interpret the inequality of a certain member of this family of spectral flows with a spectral flow induced by a Reeb flow on the boundary of the Milnor fibre as giving a sufficient condition for the 'symplectic monodromy' of the fibration to define an element of infinite order in the relative symplectic isotopy group of the Milnor fibre, this uses previous results of P. Seidel resp. of the author. We expect generalizations of the results to wider classes of (algebraic) singularities.
\end{abstract}

\section{Introduction}\label{intro}
This article continues resp. complements \cite{klein2} resp. \cite{klein3}, where we studied the Eta-invariant on the Milnor bundle of a quasihomogeneous polynomial with respect to a certain submersion metric resp. the 'symplectic monodromy' of the same bundle. One initial observation for the present work was that if $f \in \mathbb{C}[z_0,\dots,z_n]$, $n\geq 1$, is quasihomogeneous with isolated singularity in $0 \in \mathbb{C}^{n+1}$, that is there are integers $\beta_0, \dots\beta_n,\beta > 0$ such that $f(t^{\beta_0} z_0,\dots,t^{\beta_n} z_n)=t^\beta f(z_0,\dots,z_n)$ for any $t\in \mathbb{C}^*$, then there is a natural grading on the finite-dimensional $\mathbb{C}$-vectorspace 
\begin{equation}\label{mfdecomp}
M(f):=\mathcal{O}_{\mathbb{C}^{n+1},0}/(\frac{\partial f}{\partial z_0},\dots,\frac{\partial f}{\partial z_n})\mathcal{O}_{\mathbb{C}^{n+1},0}=\oplus_l M(f)_l
\end{equation}
induced by the eigenvector-decomposition corresponding to the $\mathbb{C}^*$-representation on $M(f)$ induced by the 
$\mathbb{C}^*$-action $\sigma(z)(z_0,\dots, z_n)=(z^{\beta_0}z_0,\dots,z^{\beta_n}z_n)$ on $\mathbb{C}^{n+1}$. In fact, it is well-known (\cite{bries},\cite{loo}) that there is a monomial basis $z^{\alpha(1)},\dots, z^{\alpha(\mu)}$, where $\alpha(k) \in \Lambda \subset \mathbb{N}^{n+1}, \ |\Lambda|=\mu$, of $M(f)$, such that the weights of the above representation can be chosen to be given by $l(\alpha(i))=\sum_k(\alpha(i)_k+1)w_k$, where $w_k=\beta_k/\beta$, that is $\sigma(z).z^{\alpha(i)}=z^{l(\alpha(i))}z^{\alpha(i)}$. Furthermore, considering the Milnor bundle
\[
f:f^{-1}(D)\cap B_{2n+2}=:X \rightarrow D,
\]
which is smooth over $D^*=D\setminus\{0\}$, for a small disk $D\subset \mathbb{C}$, $\mu$ equals the dimension of the middle cohomology of the fibre $F=f^{-1}(z)$ for some $z \in D^*$, which is by results of Milnor \cite{milnor} known to be $(n-1)$-times connected, parallelizable and homotopy-equivalent to a $\mu$-wedge of $n$-spheres, hence ${\rm dim}_\mathbb{C} H^n(F,\mathbb{C})=\mu={\rm dim}_{\mathbb{C}} M(f)$. It is also known, again by results of Brieskorn and Looijenga (\cite{loo}, Appendix A in \cite{klein2} for a brief introduction), that given the sheaf on $D$ $\mathcal{H}''=f_*(\Omega^{n+1}_X)/df\wedge d(f_*\Omega^{n-1}_{X/D})$, the so-called Brieskorn lattice, then the map $\phi$ sending $z^{\alpha(i)}$ to the class given by
\[
\phi(z^{\alpha(i)})=z^{\alpha(i)}dz_0\wedge\dots\wedge dz_n
\]
in $\mathcal{H}''_{0}$, defines a $\mathbb{C}$-isomorphism of vectorspaces $\phi:M(f)\simeq \mathcal{H}''_{0}/f\mathcal{H}''_{0}$, being, since $\mathcal{H}''$ is coherent, even free (\cite{bries}), also an isomorphism of the respective $\mathcal{O}_{D,0}$-modules. Again by the coherence of $\mathcal{H}''$ and the fact that $\mathcal{H}^n_{X/D}|D^*=\mathcal{H}''|D^*$, where $\mathcal{H}^n_{X/D}$ is the relative cohomology sheaf on $D$ and fixing a monomial basis of $M(f)$ as above, we conclude to get a $\mathbb{C}^*$-representation on $\mathcal{H}^n_{X/D,s}$ for any $s \in D^*$, whose (real infinitesimal) weights are exactly given by $l(\alpha(i))$ and this was a starting point for the first part of this work. The question arose: \\

{\it Is there a $\mathbb{C}^*$-action on some space parametrizing 'well-posed' boundary conditions of an elliptic differential operator defined on the Milnor bundle $Y=f^{-1}(\partial D^*)$ which is naturally induced by the above action on $\mathcal{H}^n_{X/D,s}$? What could be spectral invariants of this action and what would be their relation to the above weights $l(\alpha(i))$?}\\

To answer this question, consider a closed, symmetric (unbounded) operator $D$ on a Hilbert space $(H, (\cdot,\cdot)_H)$ with domain dense $\mathcal{D}_{min}$ and domain of the adjoint $D^*$, $\mathcal{D}_{max}$, so $\mathcal{D}_{max}$ is the maximal closed extension of $D$ in $H$. Now the space 
\[
\beta=\mathcal{D}_{max}/\mathcal{D}_{min}
\]
comes equipped with a structure of a symplectic Hilbert space with scalar product induced by the graph scalar product on $H$ and the symplectic form $\Omega$ is given by
\[
\Omega(f,g)=(Df,g)_H\ -\ (f,D^*g)_H,
\]
while there is a natural trace map $r:\mathcal{D}_{max}\rightarrow \mathcal{D}_{max}/\mathcal{D}_{min}$ being the quotient map. Then it is obvious that any domain $\mathcal{D}_{min}\subset \mathcal{D}\subset \mathcal{D}_{max}$ so that $r(\mathcal{D})$ is Lagrangian in $\beta$ defines a self-adjoint extension of $D$. Furthermore, it is known (see \cite{booblee}) that the existence of at least one self-adjoint Fredholm extension of $D$ implies that $\Lambda:=r({\rm ker}(D^*))$ is Lagrangian and the associated extension becomes Fredholm, $\Lambda$ is called the Cauchy data-space. Now consider the odd signature operator $D$ on $Y=X|\partial D$, acting on the smooth forms in $L^2(\Omega^{even}(Y,\mathbb{C}))$:
\begin{equation}\label{dirac}
D(\beta)=i^{n}(-1)^{p-1}(*d-d*)(\beta),\ {\rm for }\ \beta \in
\Omega^{2p}(Y,\mathbb{C})=:E.
\end{equation}
Assuming that $Y$ is equipped with a metric collar $[0,\epsilon)\times \partial Y$ and using Green's formula we infer that in this case $\Omega(\phi,\psi)=<r\phi,\gamma r\psi>_{\partial Y}$, where $<\cdot, \cdot>_{\partial Y}$ is the natural $L^2$-product on forms on $\partial Y$ and $\gamma^2=-id$. It is a result of Hoermander resp. Booss-Bavnbek \cite{booblee} that the Cauchy data space in this situation equals
\[
\Lambda\simeq\Lambda(D,0)=\overline {r\{f \in C^\infty(\Omega^{even}(Y,\mathbb{C}))|Df=0 \in Y\setminus\partial Y\}}^{H_{-1/2}(\Omega^*(\partial X,\mathbb{C}))},
\]
where here, $r=i^*+i^**$, where $i:\partial Y \hookrightarrow Y$, so we are naturally led to consider Lagrangian subspaces of $L^2(\Omega^*(\partial Y,\mathbb{C}))$ with respect to the symplectic form $\Omega$. Note that, compared to the above 'natural' Cauchy-data-spaces, we will work in the $L^2$-setting, which means considering the above closure in $L^2$, giving the space $\Lambda(D,1/2)$, the approaches are in a sense equivalent, as was shown in \cite{booblee}. There is a natural orthogonal pseudodifferential projection $P_Y$ of order zero onto $\Lambda$, called the Calderon projection and we will consider certain compact perturbations $P$ of this projection. More precisely, we will consider projections $P$ being pseudo-differential of order $0$ so that $(P,P_Y)$ are Fredholm pairs and whose images are Lagrangian leading to self-adjoint Fredholm extensions of $D$ with compact resolvent, as was shown by Cheeger and Bruening (\cite{brules}). We will denote the set of those projections by ${\rm Gr}(A)$, where $A$ is the tangential operator. Explicitly these self-adjoint extensions are given by
\begin{equation}\label{proj}
\mathcal{D}(D_P)=\{s \in H_1(E)|P(s_{|\partial X})=0\}\subset L^2(E), \ P\in {\rm Gr}(A).
\end{equation}
Now one observes that following Atiyah, $\mathcal{K}_Y$ the 'extended $L^2$-solutions of $D$' on the elongation $Y_\infty$ defined by 'stretching its collar to infinity' (see Appendix A) define a finite-dimensional subspace $\mathcal{K}_Y$ of $\Lambda$ (we will consider $\Lambda$ and $\mathcal{K}_Y$ in this introduction interchangingly as subspaces of $L^2(E)$ and of $L^2(E|\partial Y)$ via the trace map). On the other hand, we can perturb the metric in a neighbourhood of $\partial Y$ so that we can isometrically glue a copy of the metrically trivial bundle $Y_0=F\times S^1$ along their common boundary. The image of the Cauchy-data space of $Y_0$ under glueing will become a Lagrangian subspace of $L^2(E|\partial Y)$ which contains as isotropic closed subspace the corresponding space of extended $L^2$-solutions $\mathcal{K}_{Y_0}$ and we can represent $H^n(F;\mathbb{C})$ as a subspace of $\mathcal{K}_{Y_0}$:
\begin{equation}\label{inclusions}
H^n(F,\mathbb{C})\subset \mathcal{K}_{Y_0}\subset \Lambda.
\end{equation}
It is now clear how to extend the above $C^*$-action, restricted to $S^1$, on $H^n(F,\mathbb{C})$ to an $S^1$-action on the set of Lagrangians in $(L^2(\Omega^{*}(\partial Y,\mathbb{C})),\Omega)$ having non-trivial intersection with the symplectic subspace $\mathcal{W}$ given by $(\mathcal{K}_{Y_0}, \gamma \mathcal{K}_{Y_0})$. Let $\pi_0:\mathcal{W}\rightarrow \mathcal{K}_{Y_0}$ be the $\gamma$-orthogonal projection onto the Lagrangian subspace $\mathcal{K}_{Y_0}$ in $\mathcal{W}$. Then if $M$ is the diagonal matrix having as entries the (real) eigenvalues of the Gauss-Manin-connection $\nabla^{GM}$ acting on $\mathcal{H}^n_{X/D}$ with respect to the $\mathcal{O}_{D,0}$-basis of $\mathcal{H}''_{0} \simeq M(f)$ given by a set of $\mu$ monomials $z^{\alpha(1)},\dots, z^{\alpha(\mu)}$, we define a path of isotropic subspaces in $L^2(E|\partial Y)$ by
\begin{equation}\label{actionmatrix}
\sigma(e^{2\pi it})(s):=e^{2\pi i t\gamma (M\circ \pi_0)}s, \  s \in H^n(F,\mathbb{C}) \subset \mathcal{K}_{Y_0}.
\end{equation}
We will give a more geometric formulation of this in Chapter \ref{monodr}, which can be seen to be equivalent (Lemma \ref{actionmat}). Extending $\sigma$ by the identity to the orthogonal complement of $H^n(F,\mathbb{C})\subset \mathcal{K}_{Y_0}$ in $\Lambda$, we get a family $\Lambda(t)$ and the first result reads as (cf. Theorem \ref{theorem2}, 1.):
\begin{theorem}\label{theorem1intro}
$\Lambda(s)$ defines a family of self-adjoint Fredholm extensions of the signature operator $D$ on $Y$ and the associated spectral flow (cf. Definition \ref{S2.2}) equals the sum $-2\beta(\sum_{i=1}^\mu l(\alpha(i))-1)$, with the weights $l(\alpha(i))$ as given above.
\end{theorem}
The proof uses the fact that in the adiabatic limit, the image of the Calderon projector decomposes in a certain predictable way. This result goes back to Nicolaescu \cite{Nicol2} and Kirk and Lesch \cite{lesch} (cf. Lemma 2.14 in \cite{klein2}). Furthermore a certain boundary reduction theorem of Nicolaescu for the spectral flow is used, which involves the Maslov Index of pairs of Lagrangians of $L^2(E|\partial Y)$ (Theorem \ref{nicolaescu}). Restricting the circle action on $H^n(F,\mathbb{C})$ to single elements of a fixed basis of eigenvectors while fixing the others determines loops of Lagrangians $\Lambda_i(t)$. The corresponding family of spectral flows ${\rm SF}(\alpha(i))=\beta(l(\alpha(i))-1)$ and $\beta$ is then equivalent to the set of weights $l(\alpha(i))$ of $\sigma$ acting on $M(f)$. Since it is a result of Nemethi \cite{nem1}, that the latter determine the variation structure of $f$, given by the $4$-tupel $(H^n(F,\mathbb{C})=:U, b, h^*, V)$, where $b:U^*\rightarrow U$ is the intersection form of $f$, $h^*$ its algebraic monodromy and $V:H^n(F,\mathbb{C})\rightarrow H^n(F,\partial F, \mathbb{C}), V(\alpha)=[(h-id)^*(\alpha)]$ its variation mapping, where $h \in {\rm Diff}(F,\partial F)$ is the monodromy diffeomorphism of $f$, we get (cf. Theorem \ref{theorem2}, 3.)
\begin{folg}
The set of spectral flows ${\rm SF}(\alpha(i)), \ i=1, \dots, \mu$ and $\beta$ determines the variation structure $(U, b, h^*, V)$ of $f$.
\end{folg}
Note that from yet another viewpoint the set of spectral flows ${\rm SF}(\alpha(i)), i=1,\dots,\mu$ and the quasihomogeneous degree $\beta$ determine and are determined by the {\rm spectrum} ${\rm Sp}(f)$ of the quasihomogeneous singularity (see Varchenko \cite{varchenko2}, Definition 4.6 in \cite{klein2}), which is a set of rational numbers $\{\gamma_i\}_{i=1,\dots,\mu}$ being defined as the normalized logarithm $\gamma_i= (-1/2\pi i){\rm log} \lambda_i$ of the eigenvalues $\lambda_i, i=1\dots,\mu$ of the monodromy. Here, the normalization is determined by the asymptotic Hodge filtration on the 'canonical' Milnor fibre, for that terminology, see for instance Kulikov (\cite{kulikov}) resp. Appendix A in \cite{klein2}. In terms of our monomial basis $z^{\alpha(i)}$ of $M(f)$ and since $f$ being quasihomogeneous the monodromy is semi-simple, one has simply $\gamma_i= l(\alpha(i))-1$, hence written as a 'divisor', ${\rm Sp}(f)=\sum_{\alpha(i) \in \Lambda} \left(l(\alpha(i))-1\right) \in \mathbb{Z}^{(\mathbb{Q})}$. Now for an isolated quasihomogeneous singularity $f$ it is known that the spectrum is equivalent to the quasihomogeneous weights, while the latter determine by a result of Nemethi (\cite{nem2}) its Seifert form. By Durfee (\cite{durfee}), the Seifert form determines its topological type for $n\geq 3$, that is, the homeomorphism type of the pair $(\mathbb{C}^{n+1}, f^{-1}(0))$. Then, since the spectrum is equivalent to the $\{l(\alpha(i))\}^\mu_{i=1}$, these determine its topological type for $n\geq 3$. But Saeki (\cite{saeki}, Remark 3.10) shows that conversely, at least for $n\leq 2$, the topological type of a quasihomogeneous singularity determines the spectrum, summarizing we have
\begin{folg}\label{toptypeinv}
The set of spectral flows ${\rm SF}(\alpha(i)), \ i=1, \dots, \mu$ and $\beta$ determines and is determined by the spectrum ${\rm Sp}(f)$, that is ${\rm Sp}(f)=\sum_{\alpha(i) \in A} (-{\rm SF}(\alpha(i))/2\beta) \in \mathbb{Z}^{(\mathbb{Q})}$. Furthermore the topological type of $f$ is for $n\geq 3$, determined by and, for $n\leq 2$, determines, the set ${\rm SF}(\alpha(i)), \ i=1,\dots,\mu$.
\end{folg}
The classical viewpoint of Milnor \cite{milnor} was to describe the fibration $(Y,\partial Y)$, restricted to $S^1\subset D$ as a 'fibred knot' in $S_{2n+1}$
\begin{equation}\label{openbook}
\phi=\frac{f}{|f|}: \tilde Y:=S^{2n+1}\setminus K_f\rightarrow S^1,
\end{equation}
where $K_f=\partial f^{-1}(0)\cap S^{2n+1}$. If $F=\phi^{-1}(z_0)$ for some $z_0\in S^1$, then the Wang exact sequence of that fibration and the long exact sequence of the pair $(F,\partial F)$ on one hand together with  Alexander- resp. Poincare-duality on the other hand one gets the following commutative diagram (\cite{nem1}):
\begin{equation}\label{gysinseq}
\begin{CD}
 0 \rightarrow H^{n-1}(\partial F, \mathbb{C}) @>>{\delta}> H^{n}(F,\partial F, \mathbb{C})  @>>{b}> H^n(F,\mathbb{C}) @>>r>H^{n}(\partial F,\mathbb{C}) \rightarrow 0\\
 @VV{\simeq}V  @VV {V^{-1}}V    @VV {id} V  @VV{\simeq}V  \\
0 \rightarrow H^n(\tilde Y,\mathbb{C}) @>>> H^{n}(F, \mathbb{C})  @>>h^*-id> H^n(F,\mathbb{C}) @>>{\delta}> H^{n+1}(\tilde Y,\mathbb{C}) \rightarrow 0.\\ 
\end{CD}
\end{equation}
The philosophy is here, to relate properties of the 'link' $K_f\simeq \partial F$ to properties of the fibre $F$ resp. the fibration $Y$ and its monodromy and vice versa, i.e. we see at once that $K_f$ is a rational homology sphere if and only if $b$ is non-degenerate as a sesquilinear form which is exactly the case if the characteristic polynomial of $h^*$ at $1$ is nonzero, that is, $\Delta(1)={\rm det}(h^*-id)\neq 0$. To adopt this point, we posed the question:\\

{\it To what extent is the set of weights of $\sigma$ acting on $M(f)$, equivalently the set of spectral flows ${\rm SF}(\alpha(i))/\beta$ determined by the link resp. properties of the boundary fibration $\partial Y\rightarrow S^1$?}\\

To answer this question, note that the set of spectral flows ${\rm SF}(\alpha(i))$ intrinsically depends on the geometry of $Y$, namely the circle action $\sigma|S^1$ acting on $\mathcal{H}''_0$, however the evaluation of $\sigma$ at $e^{2\pi i 1/\beta}$ amounts to an evaluation of a certain (fibrewise) Reeb-flow over the trivial boundary fibration $S^1\times K_f$, cf. Lemma \ref{lemmageom} (note that $K_f$ is a contact manifold is a natural way, see Corollary \ref{reeb}) on the subspace $\mathcal{K}_{X_0}$. On the other hand, by a Theorem of Scott-Wojciechowski (Theorem 2.11 in \cite{klein2}) the difference of the (reduced) eta-invariants $\tilde \eta(D,\Lambda(t))$ for $D$ on $Y$ (cf. Atiyah \cite{Atiyah}) with respect to the boundary conditions given at $t=1/\beta$ resp. $t=0$ can be determined by equating this evaluation on $\mathcal{K}_{Y_0}$ over the boundary and we get (this combines Theorem \ref{theorem2} 2., Lemma \ref{lemmageom}, the remark below that Lemma and Corollary \ref{reeb}):
\begin{theorem}\label{etabla}
$\tilde \eta(D,\Lambda(1/\beta))-\tilde \eta(D,\Lambda(0))=-\sum_{i=1}^\mu \{{\rm SF}(\alpha(i))/\beta\}' +\tau$.
\end{theorem}
Here, $\{\cdot\}'$ is a certain modified fractional part and $\tau \in \mathbb{N}$ depends on $P_{Y_0}$ and its adiabatic limit and the evaluation of the Reeb-flow at $1/\beta$, note that ${\rm SF}(\alpha(i))/2\beta\in \mathbb{Z}$ if and only if the corresponding weight is an integer, that is, if $z^{\alpha(i)}$ represents an element in the kernel of $b$. So, the sum of the fractional parts of the weights $l(\alpha(i))$ resp. of the set ${\rm SF}(\alpha(i))/\beta$ are determined by the Reeb-flow on $K_f$ and the Cauchy-data space of the trivial bundle $Y_0$ only, furthermore they have a representation as a difference of eta-invariants on $Y$. Now (as mentioned above) the topological type of $f$, that is the homeomorphism type of the tuple $(\mathbb{C}^{n+1}, V_f=f^{-1}(0))$ is determined by the set ${\rm SF}(\alpha(i))$ for $n\geq 3$. This follows from the fact that the (integer) Seifert form 
\begin{equation}\label{seifertform}
S(a,b)=<V^{-1}a, b>
\end{equation}
where $<\cdot,\cdot>$ is Poincare duality, determines and is determined by the topological type and the set of weights $\{l(\alpha(i))\}_i$ determines the integer Seifert form (\cite{nem1}, \cite{varchenko2}, \cite{saeki}). Furthermore, by Theorem \ref{etabla} and Corollary \ref{toptypeinv}, we deduce:
\begin{folg}\label{folg1intro}
Assume $|{\rm SF}(\alpha(i))/\beta|<1$ for any $i=1,\dots,\mu$. Then the set of differences
\[
\eta_\Delta(i):=\tilde \eta(D,\Lambda_i(1/\beta))-\tilde \eta(D,\Lambda_i(0))\ {\rm mod}\ \mathbb{Z}, \  i=1,\dots,\mu,
\]
determines the spectrum of $f$. Hence, in this case the weights $\{l(\alpha(i))\}_i$ of $f$ and for $n\geq 3$ its topological type and Seifert form are determined by the evaluation of the Reeb-flow of $K_f$ on the restriction of $\mathcal{K}_{Y_0}$ to $K_f$, i.e. depend on the smooth Milnor fibre and the embedding of the link $K_f$ only (which is of course well-known for the Seifert form). On the other hand, if $n\leq 2$, the topological type of $f$ always determines the set $\eta_\Delta(i), \  i=1,\dots,\mu$.
\end{folg}
{\it Remark.} Since the spectral numbers $\gamma_i= l(\alpha(i))-1\ {\rm mod}\ \mathbb{Z}$ determine the algebraic monodromy $h^*$ the above implies that the set $\eta_\Delta(i)$ determines $h^*$ without any assumption on the range of the $\gamma_i$ resp. the set ${\rm SF}(\alpha(i))$. If we would chose instead of $\mathcal{H}''$ the so-called canonical extension of $\mathcal{H}^n_{X/D}|D^*$ to $D$ (cf. Kulikov \cite{kulikov}) the associated extension of the topological Gauss-Manin connection would have eigenvalues in the interval $[0,1)$ and the associated set of spectral flows would be a priori determined by boundary data as above, but of course would not be equivalent to the spectrum of $f$.\\
Finally, since the above Reeb flow acts by isometries and the foliation induced by the Reeb vector field has closed curves as its leaves one can associate to it an $S^1$-action on the set of Lagrangians in $(L^2(E|\partial Y), \Omega)$ as described in (\ref{globalproj}) and a spectral flow ${\rm SF}(\sigma_{\mathcal{B}})$ and one may ask if the latter is related to the set ${\rm SF}(\alpha(i))$ from above. As a partial answer as we will see below one has (Theorem \ref{difforder}):
\begin{theorem}
If for some $m \in \mathbb{N}$ we have $0=\rho^{\beta\cdot m}\in \pi_0({\rm Diff}(F,\partial F))$ for $\rho \in {\rm Diff}(F,\partial F)$ representing $\sigma(1/\beta)$ in $\pi_0({\rm Diff}(F))$ under the forgetful map, then $SF(\sigma_{\mathcal{B}})=0$.
\end{theorem}
Interpreting $\rho \in {\rm Diff}(F,\partial F)$ as a representative of the 'geometric monodromy' of $Y$ we see that the non-vanishing of $SF(\sigma_{\mathcal{B}})$ obstructs the ($\beta$-th power of the) geometric monodromy to be of finite order in $\pi_0({\rm Diff}(F,\partial F))$. By the results of Seidel \cite{seidel} resp. \cite{klein3} the same is true for ${\rm SF}(\alpha(1)=0)$ in the 'symplectic category', that is setting $sf(\alpha=0)=\frac{1}{2}{\rm SF}(\alpha(1)=0)$, then if $n\geq 2$ and 
\begin{equation}\label{seidelcondition}
sf(\alpha=0)=m(f)=\sum_{i=1}^{\mu}\beta_i-\beta\neq 0,
\end{equation}
the symplectic monodromy $\rho \in \pi_0({\rm Symp}(F,\partial F,\omega))$ is of infinite order (note that the condition given in \cite{klein3} is $m(f)\notin \mathbb{Z}$). Now for links of isolated hypersurface singularities $(h^*)^{\beta}=id$ and $K_f$ being a rational homology sphere are equivalent to $V(\rho^\beta)=0$ and the latter implies by \cite{krylov} that $\rho^{4\beta}=id$ in $\pi_0({\rm Diff}(F,\partial F))$, while for $n=2,6$ we have $\rho^\beta=id$, so in the latter cases we have $SF(\sigma_{\mathcal{B}})=0$. We summarize these findings in 
\begin{folg}
Let $n\geq 2$. If $SF(\sigma_{\mathcal{B}})\neq sf(\alpha=0)$, then $\rho_s \in \pi_0({\rm Symp}(F,\partial F,\omega))$ is of infinite order. If in addition, $\partial F$ is a rational homology sphere, then the map $\pi_0({\rm Symp}(F,\partial F,\omega))\rightarrow \pi_0({\rm Diff}(F,\partial F))$ has an infinite kernel.
\end{folg}
Note that in \cite{klein1} (Section 4.1/4.2), we interpret $sf(\alpha=0)$ as a winding number along the boundary of a disk which lies in a closed submanifold of $Y$ whose intersection with each Milnor fibre is Lagrangian. Inspired by this, we propose a possible scheme of proving Seidel's result resp. generalizations to arbitrary isolated algebraic singularities using the vanishing of the eta-invariant in the presence of an orientation-reversing isometry.

\tableofcontents

\section{Spectral flow and Maslov index}
In the following, we will focus briefly on the (as we will see, closely related) notions of 'spectral flow' for continuous paths of boundary value problems for a Dirac-type operator and the Maslov index on the Fredholm boundary Grassmannian of the tangential operator. We will assume familiarity with the content of Section 2 of \cite{klein2} on boundary value problems of Dirac type operators and will sometimes explicitly refer to definitions and notation in it (see also \cite{lesch} for an equivalent introduction). We first define a notion of spectral flow for paths of (not necessarily bounded) closed Fredholm operators, denoted by $\mathcal{CF}^{sa}$, on an arbitrary (separable) Hilbert space $H$ and then specialize to the needed case, the presentation will follow in essence the very clear exposition \cite{bolephi}, which focuses mainly on the use of the Cayley transform, for a presentation of spectral flow in the bounded operator context, see \cite{boos}. \\
Let now $\mathcal{C}^{sa}$ be the set of closed self-adjoint operators on $H$, we define a metric on $\mathcal{C}^{sa}$ called the {\it gap metric} $\delta(T_1,T_2)$ for $T_1,T_2 \in  \mathcal{C}^{sa}$, which is given by letting $P_j$ denote the orthogonal projections onto the graphs of $T_j$ in $H\times H$ and taking the operator norm of the difference:
\[
\delta(T_1,T_2):=||P_1-P_2||.
\]
Now the philosophy is to use the gap metric to define a topology on $\mathcal{CF}^{sa}$ that makes an appropriate map (the Cayley transform) to a certain subset of the set of unitary (i.e. bounded) operators of $H$, continuous and to use well-known techniques for defining a winding number on unitary operators to carry this notion to $\mathcal{CF}^{sa}$. Note that one advantage of using the Cayley transform, that is the map $\hat \kappa:\mathcal{CF}^{sa}\rightarrow \mathcal{U}(H)$ induced by the map 
\[
\kappa:\mathbb{R}\to S^1\setminus\{1\}, x\mapsto \frac{x-i}{x+i}
\]
in contrary of using the Riesz map 
\[
F: \mathcal{C}^{\rm sa} \longrightarrow \mathcal{B}^{\rm sa}, \ T\mapsto F_T:=T(I+T^2)^{-1/2}
\] 
onto a subset of the bounded self-adjoint operators $\mathcal{B}^{\rm sa}$ as it was done in \cite{nico2}, is that continuity in the above gap metric, for which we will prove the Cayley transform to form an homeomorphism, can be established much easier than in the Riesz metric (the one for which the Riesz map is a homeomorphism onto its image), since the Riesz topology is strictly finer than the gap topology (see \cite{bolephi}). In any case one has the result (\cite{bolephi}):
\begin{theorem}\label{S1.1} 
(a) Let $\kappa$ be as introduced above, then $\kappa$ induces a homeomorphism
\begin{equation}\label{G1.2}
\begin{split}
\hat \kappa:& \mathcal{C}^{\rm sa}(H)\longrightarrow \left \{U\in \mathcal{U}(H)|U-I {\rm\  is\ injective }\right \}=:\mathcal{U}_{\rm inj},\\
&T\mapsto \hat \kappa(T)=(T-i)(T+i)^{-1}.
\end{split}
\end{equation}
More precisely, the gap metric is (uniformly) equivalent to the metric
$\tilde \delta$ defined by $\tilde\delta(T_1,T_2)=||\hat \kappa(T_1)-\hat \kappa(T_2)||$.\\
(b) The set $\mathcal{CF}^{\rm sa}=\left\{T\in \mathcal{C}^{\rm sa}|0\not\in {\rm spec}_{\rm ess} T\right\}
=\mathcal{C}^{\rm sa}\cap \hat \kappa^{-1}(\mathcal{U}_\mathcal{F})$ where
\[
\mathcal{U}_\mathcal{F}=\left\{U\in\mathcal{U}(H)|-1\not\in {\rm spec}_{\rm ess} U\right\},
\]
of (not necessarily bounded) self-adjoint Fredholm operators is open in $\mathcal{C}^{\rm sa}$ and its Cayley image 
\[
\hat \kappa(\mathcal{CF}^{\rm sa})=\mathcal{U}_\mathcal{F}\cap \mathcal{U}_{\rm inj}=:\mathcal{U}_{{\rm inj},\mathcal{F}}
\]
is dense in $\mathcal{U}_{\mathcal{F}}$.
\end{theorem}
Note that (b) together with the fact (see \cite{bolephi}) that $\mathcal{B}^{\rm sa}$ is dense in $\mathcal{C}^{\rm sa}$ with respect to the gap metric, implies that $\mathcal{F}^{sa}$ (the bounded self-adjoint Fredholm operators on $H$) is dense in $\mathcal{CF}^{\rm sa}$ with respect to the gap metric. So in the Cayley picture, one has the following chain of dense inclusions
\begin{equation}\label{chain}
\hat \kappa(\mathcal{F}^{sa})\subset \mathcal{U}_{{\rm inj},\mathcal{F}}\subset \mathcal{U}_{\mathcal{F}}.
\end{equation}
Nevertheless, the topology of the three sets is very different, see \cite{bolephi}. Note  that the topology on $\mathcal{B}^{\rm sa}$ induced by the usual norm coincides with the topology induced by the gap metric, although the two metrics are not unitarily equivalent (see \cite{labr}). Note that in (\ref{chain}), $\mathcal{U}_{\mathcal{F}}$ is a classifying space for $K^1$ with the isomorphism $[S^1,\mathcal{U}_{\mathcal{F}}] \simeq K^1(S^1) \simeq \pi_1(\mathcal{U}_{\mathcal{F}})\simeq \mathbb{Z}$ given by the {\it winding number}, while $\mathcal{U}_{{\rm inj},\mathcal{F}}$ is not, see again \cite{bolephi}. To give the definition of spectral flow based on the Cayley transform, we need a definition of winding number. For this, we define 
\[
\mathcal{U}_{\rm tr}(H)=\{u \in \mathcal{U}(H)\mid U-I \rm{\ is\ trace\
class}\},\quad \mathcal{U}_{\mathcal{K}}(H)=\{u \in \mathcal{U}(H)\mid U-I \rm{\ is\ compact}\}.
\]
Then it is well-known that $\mathcal{U}_{\rm tr}(H)\hookrightarrow \mathcal{U}_{\mathcal{K}}(H)$ is a homotopy equivalence on the other hand it is shown in \cite{lesch}, that the inclusion $\mathcal{U}_{\mathcal{K}}(H) \rightarrow\mathcal{U}_{\mathcal{F}}(H)$ is a weak homotopy equivalence. As a consequence, giving a closed path $f:I\rightarrow \mathcal{U}_{\rm tr}(H)$ one sets
\[
{\rm wind}(f)= \frac{1}{2\pi i}\int_0^1 tr(f(t)^{-1}f'(t))dt
\]
which then extends to an group isomorphism 
\begin{equation}\label{sftrace}
{\rm wind}:\pi_1(\mathcal{U}_{\mathcal{F}}(H))\rightarrow
\mathbb{Z}.
\end{equation}
The latter definition can furthermore be (setting the appropriate conventions) extended to non-closed paths in $\mathcal{U}_{\mathcal{F}}(H)$ and leads to the following properties of ${\rm wind}$:
\begin{enumerate}
\item \textit{Path Additivity}: Let $f_1, f_2:[0,1]\to\mathcal{U}_{\mathcal{F}}(H)$
be continuous paths with $f_2(0)=f_1(1)$. Then
\[
{\rm wind}(f_1*f_2)= {\rm wind}(f_1)+{\rm wind}(f_2).
\]
\item \textit{Homotopy invariance} Let $f_1,f_2$ be continuous paths
in $\mathcal{U}_{\mathcal{F}}$. Assume that there is a homotopy
$H:[0,1]\times[0,1]\rightarrow \mathcal{U}_{\mathcal{F}}$
such that $H(0,t)=f_1(t), H(1,t)=f_2(t)$ and such that
${\rm dim \ ker}(H(s,0)+I), {\rm dim ker}(H(s,1)+I)$ are independent of $s$.
Then
${\rm wind}(f_1)={\rm wind}(f_2)$.
\item If $f:[0,1]\rightarrow \mathcal{U}_{\mathcal{F}}$ is a $C^1$--curve then
\begin{equation}
{\rm wind}(f)=\frac{1}{2\pi i}\left(\int_0^1{\rm tr}(f(t)^{-1}f'(t))dt
              -{\rm tr}({\rm log} f(1))+{\rm tr}({\rm log}f(0))\right),
\end{equation}
where the logarithm is normalized as ${\rm log}:\mathbb{C}\setminus \{0\}\rightarrow \mathbb{C}$ as
\begin{equation}\label{wind12}
 {\rm log}(re^{it})={\rm ln} \ r+i t, \quad r>0, -\pi<t\leq \pi.
\end{equation}
\end{enumerate}
After these explanations the definition of spectral flow for paths in
$\mathcal{CF}^{\rm sa}$ is straightforward:
\begin{Def}\label{S2.2} Let $f:[0,1]\rightarrow \mathcal{CF}^{\rm sa}(H)$ be a continuous
path. Then the {\em spectral flow} of $f$, ${\rm SF}(f)$ is defined by
\[
    {\rm SF}(f):={\rm wind}(\hat \kappa\circ f).
\]
\end{Def}
From the properties of $\hat \kappa$ and of the winding number one concludes
immediately:
\begin{prop}\label{S2.3}
The spectral flow ${\rm SF}$ is path additive and homotopy invariant in the following
sense: let $f_1,f_2:[0,1]\to\mathcal{CF}^{\rm sa}$ be continuous paths
and let
\[
H:[0,1]\times[0,1]\to\mathcal{F}^{\rm sa}
\]
be a homotopy such that $H(0,t)=f_1(t), H(1,t)=f_2(t)$ and such that
${\rm dim \ ker} H(s,0),$ ${\rm dim \ ker} H(s,1)$ are independent of $s$.
Then ${\rm SF}(f_1)={\rm SF}(f_2)$. In particular, ${\rm SF}$ is invariant under
homotopies leaving the endpoints fixed.
\end{prop}
To apply the above construction of spectral flow to the special case of boundary value problems for Dirac operators $D$ on a Riemannian manifold $(X,g)$ with non-empty boundary $\partial X$, we need the following Theorem, again taken from \cite{bolephi}. We assume that $D$ is of the form $D=\gamma(\frac{\partial}{\partial x}+A)$ on a metric collar of $\partial X$ (cf. (7)-(8) of Section 2.1 in \cite{klein2}, (\ref{collar}) below), $D_P$ is the self-adjoint Fredholm extension of $D$ associated to $P \in {\rm Gr}(A)$, ${\rm Gr}(A)$ the self-adjoint Fredholm Grassmannian (Definition 2.2 and Theorem 2.3 and the discussion above these in \cite{klein2}):
\begin{theorem}\label{c:grass}
For fixed $D$ the mapping 
\[
{\rm Gr}(A) \ni P \mapsto D_P \in \mathcal{CF}^{\rm sa}(L^2(X;E))
\]
is continuous from the operator norm to the gap metric.
\end{theorem}
Thus continuous paths of suitable boundary conditions for $D$ lead to a well-defined spectral flow; one can extend these arguments to the case where $D_s$ itself varies by some parameter $s$, since we will not need this case in the following we refer to \cite{bolephi}.\\
Let in the following $(H, <\ ,\ >, \gamma)$ be a Hermitian symplectic Hilbert space, that is $(H, <\ ,\ >)$ is a separable Hilbert space, $\gamma:H\rightarrow H$ is an isomorphism that satisfies $\gamma^2=-Id_H, \gamma^*=-\gamma$ and the $\pm i$-eigenspaces $\mathcal{E}_{\pm i}$ of $\gamma$ have the same dimension (infinite if $H$ is infinite, compare Definition 2.8 of \cite{lesch}). To introduce the {\it Maslov index} on pairs of paths in ${\rm Gr}(A)$, we follow essentially Lesch and Kirk \cite{lesch}, Section 6. Let $H=L^2(E|_{\partial X})$ and denote (compare Definition 2.2 in \cite{klein2})
\[
{\rm Gr}^{(2)}_{\rm Fred}(H):=\left \{(P,Q)| \ P,Q\in {\rm Gr}(H), \ (P,Q) {\rm \  are\ a\
     Fredholm\ pair}\right \},
\]
where ${\rm Gr}(H)=\{P \in \mathcal{B}(H): P=P^*, P^2=P, \gamma P\gamma^*=I-P\}$. Note that if $P,Q \in {\rm Gr}(A)$, then $(P,Q)\in {\rm Gr}^{(2)}_{\rm Fred}(H)$ (cf. Definition 2.2 of ${\rm Gr}(A)$ in \cite{klein2}). Note further that one has the diffeomorphism (setting $H=\mathcal{E}_i\oplus \mathcal{E}_{-i}$, where $\mathcal{E}_{\pm i}:={\rm ker}(\gamma\mp i)$)
\[
{\rm Gr}^{(2)}_{\rm Fred}(H)\simeq  \mathcal {U}_{\mathcal{F}}(\mathcal{E}_{-i})\times
\mathcal{U}(\mathcal{E}_i,\mathcal{E}_{-i})
\]
given by $(P,Q)\mapsto (\Phi(P)\Phi(Q)^*, \Phi(P))$, where we refer to Lemma 2.7 in \cite{klein2} or (2.7) of \cite{lesch} for the definition of the map $\Phi:{\rm Gr}(H)\rightarrow \mathcal{U}(\mathcal{E}_i,\mathcal{E}_{-i})$. $\mathcal{U}(\mathcal{E}_i,\mathcal{E}_{-i})$ are the unitary pseudodifferential isomorphisms $\mathcal{E}_i\rightarrow \mathcal{E}_{-i}$, $\mathcal{U}_{\mathcal{F}}(\mathcal{E}_{-i})$ the unitary pseudodifferential isomorphisms of $\mathcal{E}_{-i}$ so that $-1$ is not in the essential spectrum, (cf. (13) of \cite{klein2}). Then the following definition can also be applied to finite-dimensional hermitian vector spaces, i.e. $H={\rm ker}\ A$:
\begin{Def}\label{MaslovWinding} For a continuous path $(f,g)$ in
${\rm Gr}^{(2)}_{\rm Fred}(H)$ we define the Maslov index as related to the winding number by the
equation
\begin{equation}\label{maslowind}
     {\rm Mas}(f,g)=-{\rm wind}(\Phi(f)\Phi(g)^*).
\end{equation}
\end{Def}
Using the picture of Lagrangian subspaces (note the fact that $(L_1,L_2)$ is Fredholm (resp. invertible)
if and only if the pair of projections $(I-P_{L_1},P_{L_2})$ is Fredholm (resp. invertible)), let
$(f,g):[0,1]\to{\rm Gr}^{(2)}_{\rm Fred}(H)$ be a continuous path, then the Maslov index ${\rm Mas} (f,g)$ is the algebraic count of how many times ${\rm ker}\ f(t)$ passes through ${\rm im}\ g(t)$ along the path.
Now it is relatively straightforward to prove (see \cite{lesch}) that the Maslov index obeys the following properties which actually {\it define} the Maslov index (cf. \cite{Nicol2}, \cite{CapLeeMil},\cite{CapLeeMil2}):
\begin{enumerate}
\item {\it Path Additivity:} Let $(f_j,g_j):[0,1]\rightarrow
{\rm Gr}^{(2)}_{\rm Fred}(H), j=1,2$,
be continuous paths with $f_2(0)=f_1(1), g_2(0)=g_1(1)$ then
\[ {\rm Mas}((f_1,g_1)*(f_2,g_2))={\rm Mas}(f_1,g_1)+{\rm Mas}(f_2,g_2).\]
\item {\it Homotopy Invariance:} Let $(f_j,g_j):[0,1]\rightarrow
{\rm Gr}^{(2)}_{\rm Fred}(H)$, $j=0,1$, such that $(f_0,g_0)$ is homotopic
$(f_1,g_1)$ relative  endpoints then
\[ 
{\rm Mas}(f_0,g_0)={\rm Mas}(f_1,g_1).
\]
More generally, suppose that $(F,G)$ is a homotopy so that the second parameter is the deformation parameter and so that ${\rm dim}({\rm ker} \ F(0,s)\cap {\rm im}\ G(0,s))$ and ${\rm dim}({\rm ker}\  F(1,s)\cap {\rm im}\ G(1,s))$ are independent of $s\in [0,1]$. Then $ {\rm Mas}(f_0,g_0)={\rm Mas}(f_1,g_1).$
\item {\it Normalization:} On one hand one finds that on paths with endpoints in the set of {\it invertible} pairs of projections ${\rm Gr}^{(2)}_*(H)\subset {\rm Gr}^{(2)}_{\rm Fred}(H)$ the Maslov index
induces a group isomorphism $\pi_1({\rm Gr}^{(2)}_{\rm Fred}(H),{\rm Gr}^{(2)}_*(H))\rightarrow \mathbb{Z}$, this in turn determines ${\rm Mas}$ on paths with endpoints in ${\rm Gr}^{(2)}_*(H)$ up to a sign. The sign-convention is as follows: if $(P,Q)\in {\rm Gr}^{(2)}_{\rm Fred}(H)$ then ${\rm Mas}(e^{t\gamma }Pe^{-t\gamma },Q)_{-\epsilon\leq t\leq \epsilon}={\rm dim} ({\rm ker}\ P\cap{\rm im}\ Q)$ for $\epsilon$ small enough. On the other hand if $(f,g):[0,1]\rightarrow {\rm Gr}^{(2)}_{\rm Fred}(H)$ is an arbitrary continuous
path then one can perturb the paths so that they become invertible paris at the end-points so that the forgoing convention can be applied: choose $\epsilon$ small enough such that the pairs $(e^{s\gamma }f(j)e^{-s\gamma },g(j))$ are invertible for $j=0,1, 0<s\leq\epsilon$. Then
\begin{equation} {\rm Mas}(f,g)= -{\rm wind}(\Phi(f)\Phi(g)^*)={\rm Mas}(e^{\epsilon \gamma }fe^{-\epsilon \gamma
},g).\end{equation}
\end{enumerate}
Before closing this section we give a version of a Maslov triple index in the special case which is needed in this article. For this note again that for $P\in {\rm Gr}(A)$ one has $-1 \notin {\rm spec}_{\rm ess}(\Phi(P)\Phi(P_X)^*)$, thus $-1$ is an isolated point in the spectrum of $\Phi(P)\Phi(P_X)^*$, so that we can choose a holomorphic branch of the logarithm which coincides on ${\rm spec}(\Phi(P)\Phi(P_X)^*)$ with ${\rm log}:\mathbb{C}\setminus \{0\}\rightarrow \mathbb{C}$ defined (already above) as
\[
 {\rm log}(re^{it})={\rm ln\ r}+i t, \quad r>0, -\pi<t\leq \pi.
\]
The form of the triple index given below is actually a consequence of a more general definition involving a certain 'double index' (see \cite{lesch}), we take the following as a definition for the case $P,Q,R\in{\rm Gr}(A)$ such that $P-Q,Q-R$ are trace class.
\begin{Def}\label{tripleindex}
Let $P,Q,R\in{\rm Gr}(A)$ such that $P-Q,Q-R$ are trace class. Then one defines the {\it Maslov triple index} of $P,Q,R$ as
\begin{equation}
\begin{split}
   \tau_\mu(P,Q,R)=\frac{1}{2\pi i}\bigl(&{\rm tr}\ {\rm log}(\Phi(P)\Phi(Q)^*)+{\rm tr}\ {\rm log}(\Phi(Q)\Phi(R)^*)\\
&-{\rm tr}\ {\rm log}(\Phi(P)\Phi(R)^*)\bigr).
 \end{split}
\end{equation}
\end{Def}
The next property will be important below: the homotopy invariance of the triple index.
\begin{lemma}[\cite{lesch}]\label{triplehomotopy}
Let $P,Q,R:[0,1]\rightarrow  {\rm Gr}(H)$ be paths in ${\rm Gr}(H)$ so that $(P,Q), (Q,R), (P,R)$ map into ${\rm Gr}_{\rm Fred}^{(2)}(H)$ and two of the differences $P-Q, Q-R, P-R$ are trace class. Suppose further that ${\rm dim}({\rm ker} \ P(t)\cap {\rm im}\ Q(t))$, ${\rm dim}({\rm ker} Q(t)\cap {\rm im}\ R(t))$, and ${\rm dim}({\rm ker} \ P(t)\cap {\rm im}\ R(t))$ are independent of $t$. Then
\[ 
\tau_\mu(P(0),Q(0),R(0))= \tau_\mu(P(1),Q(1),R(1)).
\]
\end{lemma}
\begin{proof} As is shown in \cite{lesch} one has the following 'cocycle property' of the Maslov index
\begin{equation}\begin{split}
    \tau_\mu&(P(0),Q(0),R(0))-\tau_\mu(P(1),Q(1),R(1))\\
      &={\rm Mas}(P,Q)+{\rm Mas}(Q,R)-{\rm Mas}(P,R).
        \end{split}
\end{equation}
Then the claim follows immediately from the homotopy invariance of the Maslov index.
\end{proof}
To close this section we cite a theorem relating Maslov index and spectral flow which is originally due to Nicolaescu \cite{Nicol2} but which we give in the slightly generalized form proven by Kirk and Lesch in (\cite{lesch}, Theorem 7.5):
 \begin{theorem}\label{nicolaescu} Let $X$ be a manifold with boundary  and $D(t),
a\leq t\leq b,$ a smooth family of Dirac operators. We assume that in a collar of
the boundary $D$ takes the form $\gamma(\frac{d}{dx}+A(t))$ as before. Let
$P(t)\in{\rm Gr}(A(t))$ be a smooth family. Denote by $P_X(t)$ the Calderon projectors of $D(t)$, and
$L_X(t)={\rm im} P_X(t)$ the Cauchy data spaces. Then
\[ {\rm SF}(D_{P(t)}(t))_{t\in [a,b]}=  {\rm Mas}(P(t),P_X(t))_{t\in[a,b]}= {\rm Mas}({\rm ker}
P(t), L_X(t))_{t\in [a,b]}.\]
\end{theorem}
Note that it is essential here that $\gamma$ is assumed to be constant.

\section{Variation structure and Spectral flow}\label{monodr}
Let  $f \in \mathbb{C}[z_0,\dots,z_n]$ be quasihomogeneous, i.e. there are integers $\beta_0, \dots\beta_n, \beta > 0$ such that $f(z^{\beta_0} z_0,\dots,z^{\beta_n} z_n)=z^\beta f(z_0,\dots,z_n)$, the weighted circle action $\sigma$ on $\mathbb{C}^{n+1}$
given by
\begin{equation}\label{weightedfirst}
\sigma(t)(z_0,\dots, z_n)=(e^{2\pi i t\beta_0}z_0,\dots,e^{2\pi i t\beta_n}z_n)
\end{equation}
preserves the Milnor fibres, its $1/\beta$-evaluation is isotopic to the geometric monodromy $g$ of the fibration. As mentioned in the introduction (cf. also \cite{klein3}), considering the Milnor fibration of $f$ with Milnor fibre $M$ (see below) as a symplectic fibration, Seidel \cite{seidel} shows:
\begin{theorem}\label{seidelthm}
For $n \geq 2$ the symplectic monodromy $\rho$ defines an element of infinite order in $\pi_0({\rm Symp}(M,\partial M,\omega))$ if the sum of the weights $w_i=\beta_i/\beta$ is not equal to one.
\end{theorem}
Note that the condition $\sum_{i=1}^\mu \beta_i/\beta\neq 1$ is a sufficient but not necessary condition, as Seidel \cite{seidel} shows. An important tool in the proof of the above is the definition of a Maslov-type number $m_f$ (see \ref{seidelcondition}) associated to a quasihomogeneous polynomial $f$. This number counts loosely speaking the number of times (counting dimension and sign) of which the image of an arbitrary Lagrangian $\Lambda \subset T_x{\partial M}, x \in \partial M$ under the weighted circle action intersects a fixed Lagrangian (using the trivialization induced by $\mathbb{C}^{n+1}$). More precisely it is the evaluation of the Maslov class $C(TM,\omega) \in H^1(\mathcal{L}(TM,\omega),\mathbb{Z})$ on this path of Lagrangians, here $\mathcal{L}(TM,\omega)$ is the Lagrangian Grassmannian of $TM$, it is shown that $m_f=\sum_i w_i \ - 1$, so the sufficient condition in the Theorem is $m_f \neq 0$.\\
In the following we will show that this number can also be obtained as a certain spectral flow of the signature operator under a certain variation of boundary conditions by using 'period mappings'. More specifically, consider the Milnor fibration (\cite{milnor}) of a quasihomogeneous polynomial
\[
f:  X=\bigcup_{u \in S^1_\delta}X_u:=f^{-1}(u) \cap B^{2n+2}\rightarrow S^1_\delta,
\]
for $\delta >0$ sufficiently small, where $S^1_\delta=\{x\in \mathbb{C}^*: ||x||=\delta\}$. We now have (compare this to Lemma 3.3 in \cite{klein2}, where it is assumed that $X$ carries a submersion metric):
\begin{lemma}
The smooth family of manifolds $\partial X=\{X_u\cap S^{2n+1}, u \in D^*_\delta\}$ admits a trivialization which is unique up to homotopy.
\end{lemma}
Since in fact a neighbourhood of $\partial X$ in $X$ extends to a smooth fibration over $D_\delta$, which is contractible, we can assume that for some fixed fibre $X_u$, there are open subsets $\partial X \subset U \subset X$, $\partial X_u \subset V \subset X_u$ and a diffeomorphism $\Theta:U \rightarrow V \times S^1$, so that $\Theta$ respects the fibred structure, that is $\Theta(U \cap X_z)= V \times \{z\}$ for $z \in S^1$, maps boundaries to boundaries
\[
\Theta|\partial X:\partial X \rightarrow \partial X_u \times S^1 \ {\rm is \ a \ diffeomorphism},
\]
and satisfies $\Theta|_{U \cap X_u}=id_V$. We equip $X_u$ with the metric induced by the restriction of the Euclidean metric in $\mathbb{C}^{n+1}$ and assume that the metric $g^X$ on $U \subset X$ satisfies
\begin{equation}\label{dec}
(\Theta^{-1})^*(g^X|U)=g^{S^1_\delta} \oplus g^{X_u}|{V}.
\end{equation}
Here the splitting is orthogonal, $g^{S^1_\delta}$ is the standard metric on $S^1_\delta$ and we assume there is an open set $W\subset X-U$, so that the metric $g^X$ of $X$, restricted to $W$, is again the metric induced from the Euclidean metric in $\mathbb{C}^{n+1}$ (using an appropriate partition of unity). Finally we assume that we have an orthogonal decomposition
\begin{equation}\label{collar}
g^{X_u}|{V} \simeq dr^2 + g^{\partial X_u},
\end{equation}
where $r \in[0,\epsilon)$ is some collar coordinate in $V$, by (\ref{dec}) this induces a metric collar on $U \subset X$.
In other words, we have a fibration $X$ and an isometry $\Theta$ of a neighbourhood of the boundary of $X$ to some metric product fibration over $S^1$ which preserves boundaries and metric collars and is the identity on a fixed fibre (i.e. we assume $g^X|_{X_u \subset X}$ to coincide with the metric induced from its embedding in $\mathbb{C}^{n+1}$).\\
Let $D$ be the signature operator on $X$ with respect to $g^X$, $D$ is of the form $D=\gamma(\frac{\partial}{\partial x}+A)$ on the metric collar of $\partial X$ (cf. Section 2.2 of \cite{klein2}). Let $D_P$ be the self-adjoint Fredholm extension associated to a projection $P \in {\rm Gr}(A)$ (Definition 2.2 in \cite{klein2}). Let $L_X={\rm im}(P_X)$ be the Cauchy data space as the image of the Calderon projector $P_X \in {\rm Gr}(A)$ of $X$ as defined in the discussion below Theorem 2.4 in \cite{klein2} and let $L^\infty_X=\lim_{r \rightarrow \infty}L^r_X$ be its 'adiabatic' limit when stretching the length of the collar to infinity (Theorem 2.16 in \cite{klein2}, Theorem 8.5 in \cite{lesch}). Let $E_\mu$ be the $\mu$-eigenspace of $A$ and recall the notation
\begin{equation}\label{decomp431}
\begin{split}
E_\nu^+&=\oplus_{0<\mu \leq \nu} E_\mu,\quad E_\nu^-=\oplus_{-\nu \leq \mu  < 0} E_\mu,\\
F_\nu^+&= \oplus_{\mu > \nu} E_\mu, \quad F_\nu^-= \oplus_{\mu < -\nu} E_\mu,
\end{split}
\end{equation}
so $L^2(\Omega^*_{\partial X})= F_\nu^-\oplus E_\nu^-\oplus{\rm ker}\ A \oplus E_\nu^+ \oplus F_\nu^+ $. Recall that if $\nu$ is any number great enough so that $L_X \cap F_\nu ^-=0$ (the smallest such $\nu$ called the {\it non-resonance level}), Nicolaescu (\cite{nico}, Theorem 4.9, compare Theorem 8.5 in \cite{lesch}) shows that the limit of the family of Cauchy data spaces $L^r_X$ exists and decomposes as
\begin{equation}
L^\infty_X=\left(\lim_{r \rightarrow \infty}e^{rA}R_\nu(L_X)\right ) \oplus F_\nu^+ =: \Lambda_X \oplus F_\nu^+.
\end{equation}
Here, note that $E_\nu^-\oplus {\rm ker}\oplus E_\nu^+=:V_\nu$ is a finite dimensional symplectic subspace of  $L^2(\Omega^*_{\partial X})$,
\[
R_\nu(L_X)=\frac{L_X\cap(F_\nu^-\oplus E_\nu^-\oplus {\rm ker\ A}\oplus E_\nu^+)}{L_X\cap F_\nu^-}
\]
is the symplectic reduction of $L_X$ with respect to $F_\nu ^-$. Then it was stated in Theorem 2.16 of \cite{klein2} (resp. Theorem 8.5 in \cite{lesch}) that there is an (isotropic) subspace $W\subset \Lambda_X$ so that 
\begin{equation}\label{isom}
W \simeq  {\rm im}(H^{even}(X,\partial X,\mathbb{C})\rightarrow H^{even}(X,\mathbb{C})) \oplus {\rm im}(H^*(X,\mathbb{C})\rightarrow H^*(\partial X, \mathbb{C})),
\end{equation}
where the isomorphism is smooth (restricting extended $L^2$-harmonic forms) and dependent on the metric on $X$ and $\partial X$. Now fix a fibre $X_u=:F$ of $X$ and set 
\[
X_0=-F \times S^1_\delta, 
\]
with the induced product metric and so that $(F,-\partial F)$ carries the 'opposite' orientation of $(F,\partial F)$ and there is, analogous to (\ref{dec}) an orientation reversing isometry $\tilde \Theta$ which identifies $[0,\epsilon)\times \partial X_0\simeq (-\epsilon,0]\times \partial X$. Define the Cauchy data space $L_{X_0} \subset L^2(\Omega^*_{\partial X_0})$, resp. $\tilde W \subset \Lambda_{X_0}^{\infty}$ analogous as for $X$, with $P_{L_{X_0}}$ s.t. ${\rm im}(P_{L_{X_0}})=L_{X_0}$ and write $L_{X_0}^\infty=\Lambda_{X_0} \oplus F_\nu^-$.  Here and in the following we choose $\nu$ greater than the maximum of the non-resonance-levels of $X$ and $X_0$. \\
Finally, fix representing monomials
\[
\{z^\alpha : \alpha \in \Lambda \subset
\mathbb{N}^{n+1}, |\Lambda|= \mu\}\ {\rm s. t.}\ \{[z^\alpha]\}_\alpha \ {\rm spans}\ 
\mathcal{O}_{\mathbb{C}^{n+1},0}/\rm{grad}(f)\mathcal{O}_{\mathbb{C}^{n+1},0}=:M(f)
\]
where dim $M(f)=\mu$, write $l(\alpha)=\sum_{i=0}^n(\alpha_i+1)w_i$. To each monomial $z^{\alpha},\ \alpha \in \Lambda$ there is an associated section $\Phi_{\alpha}(u), \ u \in D^*_\delta$ of $\mathcal{H}^n(f_*\Omega^\cdot_{X/D^*_\delta})$ (cf. Appendix A of \cite{klein2}, \cite{loo}), denote the set of these sections by $\Gamma_\Lambda$. Any $\Phi_\alpha$ defines a cohomology class $[\phi_{\alpha}] := ev_u(\Phi_\alpha)$ of $H^n(X_u,\mathbb{C})$ for any $u \in D^*_\delta$ (and $\delta$ small enough), where $ev_u$ is the evaluation mapping on the fibre $X_u$. The set $\{\phi_{\alpha}\}_{\alpha \in \Lambda}$ restricts to a basis of $H^n(\tilde X_u,\mathbb{C})$ for any $u \in D^*_\delta$ and diagonalizes the intersection form in at least a fixed fibre $X_u$. Fix such a fibre in the following, then the weighted circle action $\sigma$ on $\mathbb{C}^{n+1}$ induces a circle action $\tilde \sigma$ on $H^n(X_u,\mathbb{C})$ by requiring that for any $\alpha \in \Lambda$ the diagram 
\begin{equation}
\begin{CD}
\Gamma_\Lambda \subset \Gamma(\mathcal{H}^n(f_*\Omega^\cdot_{X/D^*_\delta})) @>>{ev_u}> H^n(X_u,\mathbb{C}) \\
 @VV {\sigma^*} V    @VV {\sigma} V \\
\Gamma(\mathcal{H}^n(f_*\Omega^\cdot_{X/D^*_\delta}))  @>>{ev_u} > H^n(X_u,\mathbb{C}) \\ 
\end{CD}
\end{equation}
commutes, that is
\begin{equation}\label{circleact34}
\begin{split}
\sigma:& S^1\times H^{n}(X_u,\mathbb{C})\rightarrow H^{n}(X_u,\mathbb{C})\\
\sigma(t)&(\phi_{\alpha})=ev_{u}\circ (\sigma^*(t)\Phi_{\alpha}),\quad \alpha \in \Lambda\subset \mathbb{N}^{n+1}.
\end{split}
\end{equation}
Now to any class $[\phi_\alpha] \in H^n(X_u,\mathbb{C}), \ \alpha \in \Lambda$ we associate a mapping $\sigma_\alpha:S^1 \times H^n(X_u,\mathbb{C})\rightarrow H^n(X_u,\mathbb{C})$ by setting
\begin{equation}\label{picardlef}
\sigma_\alpha(t)=\sigma(t)\circ P_{\alpha}+(I-P_{\alpha}),
\end{equation}
where $P_\alpha$ for any $\phi_\alpha \in H^n(X_u,\mathbb{C})$ is the orthogonal projection onto the subspace spanned by $[\phi_\alpha]$ in $H^n(X_u, \mathbb{C})$ and $\sigma(\cdot)$ is as defined in \ref{circleact34}. Note that here, the orthogonality is defined with respect to the $L^2$-inner product $(\cdot,\cdot)$ on $X_u$, so if $<,>$ denotes the (non-degenerate) Poincare duality pairing then
\begin{equation}\label{pairing}
(\omega,\alpha)= <\omega, * \alpha >=\int_{X_{u}}\omega\wedge * \alpha, \ \alpha \in H^n(X_u,\partial  X_u,\mathbb{C}), \omega \in H^n(X_u,\mathbb{C}),
\end{equation}
note that it is $*$, the Hodge star on $X_u$ with respect to its (induced) metric, which makes the pairing non-degenerate. Furthermore, note that $\sigma=\prod_{\alpha\in \Lambda}\sigma_\alpha$, where $\prod$ refers to multiplication of matrices. More explicitly, we associate for any $t \in S^1$ and any $\alpha \in \Lambda$ a mapping on $H^n(X_u,\mathbb{C})$ by setting
\begin{equation}\label{action35}
\begin{split}
\sigma_\alpha :& S^1\times H^{n}(X_u,\mathbb{C})\rightarrow H^{n}(X_u,\mathbb{C})\\
\sigma_\alpha (t)&(\omega)=\sigma(t)(\phi_\alpha)(\phi_\alpha,\omega) +\left(\omega-(\phi_\alpha,\omega)\phi_\alpha\right).
\end{split}
\end{equation}
Here we assumed that $\phi_\alpha|X_u$ is normalized with respect to $(\cdot,\cdot)$. In the following we will frequently make use of the fact that the set $\{\phi_\alpha\}_{\alpha \in \Lambda}$ can be chosen so that it orthonormalizes $(\cdot,\cdot)$ (over at least one point $u \in D^*_\delta$), a proof of this is given in Lemma 3.19 of \cite{klein2}.\\
We can give an equivalent interpretation of (\ref{picardlef}) resp. (\ref{action35}) using parallel transport in the sheaf $\mathcal{H}^n(f_*\Omega^\cdot_{X/D^*_\delta})$ (with respect to the Gauss-Manin-connection, see \cite{loo}). For this, let $X^\beta$ be the $\beta$-fold cyclic covering of $X$ (see also below), i.e. global parallel sections are well-defined. Then there is a 'relative' version of (\ref{picardlef}). For this let for any global section $\Phi_\alpha  \in \Gamma_\Lambda \subset \Gamma(\mathcal{H}^n(f_*\Omega^\cdot_{X^\beta/D^*_\delta}))$ be $\Phi_{\alpha,||}$ the associated parallel section $\Phi_{\alpha,||}\in \Gamma(\mathcal{H}^n(f_*\Omega^\cdot_{X^\beta/D^*_\delta}))$ coinciding with $\Phi_\alpha$ over the fixed fibre $X_u$, denote the set of these sections by $\Gamma_{||, \Lambda}$. Then define the mapping
\begin{equation}
{\tilde \sigma}:\Gamma_{||, \Lambda}\rightarrow \Gamma_{\Lambda}\quad {\rm by}\quad {\tilde \sigma}(\Phi_{\alpha,||})=\Phi_{\alpha}
\end{equation}
and define furthermore for any $\alpha \in \Lambda$
\begin{equation}\label{relpic}
\begin{split}
&{\tilde \sigma}_\alpha :\Gamma(\mathcal{H}^n(f_*\Omega^\cdot_{X^\beta/D^*_\delta}))\rightarrow \Gamma(\mathcal{H}^n(f_*\Omega^\cdot_{X^\beta/D^*_\delta}))\\
&{\tilde \sigma}_\alpha({\bf \omega})={\tilde \sigma}\circ P_{\alpha,||}({\bf \omega})+(I-P_{\alpha,||})({\bf \omega}),
\end{split}
\end{equation}
where ${\bf \omega}$ is any element of $\Gamma(\mathcal{H}^n(f_*\Omega^\cdot_{X^\beta/D^*_\delta}))$ and $P_{\alpha,||}$ is the fibrewise orthogonal projection onto the subspace spanned by the fibrewise restriction of $\Phi_{\alpha,||}$ with respect to the fibrewise $L^2$ inner product $(\cdot,\cdot)_x$ on each $X_x, \ x \in D_\delta^*$ as above. Understanding this, one has:
\begin{lemma}\label{par}
For $u\in D^*_\delta$, let $\mathcal{P}_{t}:\mathcal{H}^n(f_*\Omega^\cdot_{X^\beta/D^*_\delta})_u \rightarrow \mathcal{H}^n(f_*\Omega^\cdot_{X^\beta/D^*_\delta})_{e^{2\pi it}u}, t\in [0,1]$ be the parallel transport along $c(t)=e^{2\pi it}u$, then 
$\sigma_\alpha(t): S^1\times H^{n}(X_u,\mathbb{C})\rightarrow H^{n}(X_u,\mathbb{C})$ as defined in (\ref{action35}) is given by (with $\Phi_\alpha, \ \alpha \in \Lambda$ as above)
\[
\sigma_\alpha(t)(\omega)=\left((\mathcal{P}_{t})^{-1}\circ{\tilde  \sigma}_{\alpha}|_{X_{c(t)}} \circ \mathcal{P}_{t}\right)(\omega), \quad t \in [0,1], 
\]
for any $\omega \in H^{n}(X_u,\mathbb{C})$, so we have the commuting diagram for any $t \in[0,1]$:
\[
\begin{CD}
\mathcal{H}^n(f_*\Omega^\cdot_{X/D^*_\delta})_u @>>{\mathcal{P}_{t}}>  \mathcal{H}^n(f_*\Omega^\cdot_{X/D^*_\delta})_{c(t)}\\
 @VV {\sigma_\alpha(t)} V    @VV {{\tilde \sigma}_{\alpha}}|_{X_{c(t)}} V \\
\mathcal{H}^n(f_*\Omega^\cdot_{X/D^*_\delta})_u  @>>{\mathcal{P}_{t}}>  \mathcal{H}^n(f_*\Omega^\cdot_{X/D^*_\delta})_{c(t)}\\ 
\end{CD}
\]
\end{lemma}
\begin{proof}
The proof is postponed until Lemma \ref{mult} is established.
\end{proof}
We now arrive at 
\begin{Def}\label{lagrangian0}
For each monomial $z^\alpha, \alpha \in \Lambda$, $|\Lambda|=\mu$ we associate a mapping
\[
\overline \sigma_{\alpha}:S^1 \times H^*(X_0,\mathbb{C})\rightarrow H^*(X_0,\mathbb{C})
\]
by requiring that the diagram
\begin{equation}\label{act}
\begin{CD}
H^{0,n}(F,\mathbb{C})\otimes H^*(S^1,\mathbb{C}) @>>\simeq> H^*(X_0,\mathbb{C})\\
 @VV {(id\oplus \sigma_\alpha(\cdot))\otimes id} V    @VV {\overline \sigma_\alpha} V \\
H^{0,n}(F,\mathbb{C}) \otimes H^*(S^1,\mathbb{C})  @>>\simeq> H^*(X_0,\mathbb{C})\\ 
\end{CD}
\end{equation}
commutes (here and in the following we use the (unusual) notation $H^{0,n}(F_u,\mathbb{C}):=H^0(F_u,\mathbb{C})\oplus H^n(F_u,\mathbb{C})$).
\end{Def}
Following the identification (\ref{isom}), applied to $X_0$, we see that $\tilde W \subset \Lambda_{X_0}$ can be viewed as a subspace of $H^*(X_0,\mathbb{C})$ by identifying it with the space of extended $L^2$-harmonic forms on the elongation $X_{0,\infty}$ (see also Corollary \ref{harmonicbla} in Appendix A). Furthermore we will see that $\overline \sigma_\alpha(r^{-1}(\tilde W))\subset r^{-1}(\tilde W)$. To ensure that (\ref{isom}) preserves orthogonality of $L^2$-products, which will be sufficient for the action on $L^2(\Omega^*(\partial X_0,\mathbb{C}))$ induced by (\ref{act}) as defined below (see (\ref{action1})) to be symplectic, we will compose the restriction $r:r^{-1}(\tilde W) \rightarrow \tilde W$ by the map $\kappa:=\sqrt{rr^*}^{-1}:\tilde W \rightarrow \tilde W$, so that 
\begin{equation}
\mathcal{L}:= (\kappa\circ r): r^{-1}(\tilde W)\subset H^*(X_0,\mathbb{C}) \rightarrow \tilde W \ {\rm is \ a\ linear \ isometry}
\end{equation}
with respect to the respective $L^2$ inner products. For the following definition, let $\overline \sigma_{\Lambda'}$ equal any product $\prod_{\alpha \in \Lambda' \subset \Lambda}\overline \sigma_\alpha$ for an arbitrary subset $\Lambda'\subset \Lambda$, note that the $\overline \sigma_\alpha, \alpha \in \Lambda$ commute as a result of Lemma \ref{mult} below and Lemma 3.19 in \cite{klein2}, which states the orthogonality of the $\{\phi_\alpha\}_{\alpha \in \Lambda}$ wrt the non-degenerate pairing (\ref{pairing}) (we will suppress the index $\Lambda'$ in the definition occasionally):
\begin{Def}\label{def}
Define a family of self-adjoint extensions $D_{P_{\Lambda'}(t)}$ of the signature operator $D$ on $X$ by composing the following three paths to a path $P_{\Lambda'}(t) \in {\rm Gr}(A), t \in [-1,2]$:
\[
\begin{split}
 1.\quad P_1(t)&=I - P_{L_{X_0}^{1/-t}}, \ t \in [-1,0],\\
 2.\quad P_2(t)&= I - P_{L_{X_0,t,\Lambda'}^\infty},\ t \in [1,2],\\
 3.\quad P_3(t)&= I-P_{L_{X_0}^{1/(t-1)}}, t \in [1,2].
\end{split}
\]
Here $L_{X_0,\Lambda'}^\infty(t), t \in [0,1]$ is a path of Lagrangians in $L^2(\Omega^*_{\partial X})$ defined as follows:\\
Let $\phi(\Lambda_{X_0}): ker(\gamma -i)\cap V_\nu \rightarrow {\rm ker}(\gamma+i)\cap V_\nu$ be the isometry associated to $\Lambda_{X_0}$. Then write $\Lambda_{X_0}=\tilde W \oplus \tilde W'$, note that ${\rm ker}(\gamma\pm i)\cap V_\nu = (I\pm i\gamma)\Lambda_{X_0}$ since $(I\pm i\gamma):\Lambda_{X_0}\rightarrow {\rm ker}(\gamma \pm i)$ is an isomorphism. Define a circle action on the $\mp i$-eigenspaces of $\gamma$ restricted to $V_\nu$ 
\[
\hat \sigma:S^1 \times ({\rm ker}(\gamma\pm i)\cap V_\nu =(I \pm i\gamma)(\tilde W \oplus \tilde W')) \rightarrow {\rm ker}(\gamma \pm i)\cap V_\nu,
\]
by 
\begin{equation}\label{action1}
\begin{split}
\hat \sigma(t)(I+ i\gamma)(x+y)&=(I + i\gamma)(\mathcal{L}\overline \sigma(t)\mathcal{L}^{-1}(x)+y),\\
\hat \sigma(t)(I - i\gamma)(x+y)&=(I - i\gamma)(\mathcal{L}\overline\sigma(t)^*\mathcal{L}^{-1}(x)+y),
\end{split}
\ {\rm for}\ x \in \tilde W, y \in \tilde W',
\end{equation}
where $(\cdot)^*$ means the adjoint with respect to the $L^2$-inner product restricted to $H^*(X_0,\mathbb{C})$.
Define finally
\begin{equation}\label{action}
L_{X_0,\Lambda'}^\infty(t) =\left(\hat \sigma(t)\left(x+ \phi (\Lambda_{X_0})x\right), \ x \in V_\nu \cap {\rm ker\ }(\gamma-i)\right )\oplus F_\nu^-.
\end{equation}
\end{Def}
Note that via the restriction map $r:\oplus_p\Omega^{2p}(X_0,\mathbb{C})\rightarrow \oplus_k \Omega^{k}(\partial X_0,\mathbb{C})$ by Theorem 2.14 in \cite{klein2} the above subspace $\tilde W \subset L^2(\Omega^*(\partial X_0,\mathbb{C}))$ is isomorphic to the 'space of extended $L^2$-solutions of $D\beta=0$ on $X_{0,\infty}$' in the sense of Atiyah (\cite{Atiyah}) (here, $X_{0,\infty}$ is the manifold obtained from $X_0$ by attaching an infinite collar), so
\[
\mathcal{K}_{X_0}:=r^{-1}(\tilde W)=\{\beta \in \Omega^{even}_{X_0}\  {\rm s.t.}\  D\beta=0\ {\rm and}\ r(\beta) \in F_0^-\oplus{\rm ker}(A)\}.
\]
Then $\mathcal{K}_{X_0}\subset r^{-1}(L_{X_0})$ is by (\cite{Atiyah}) isomorphic to the right hand side of (\ref{isom}), which gives for $\Lambda'\subset \Lambda$ a {\it smooth} circle action 
\begin{equation}\label{circleact3}
\overline \sigma_{\Lambda'}(t): \mathcal{K}_{X_0} \rightarrow \mathcal{K}_{X_0}, \quad t \in S^1,
\end{equation}
as induced by (\ref{act}) using the fact that $X_0$ is a (metric) product. That $\overline \sigma_{\Lambda'}$ actually acts in a smooth way on $\mathcal{K}_{X_0}\simeq \tilde W$, follows since to any fixed de Rham-basis of $H^n(F,\mathbb{C})$ one can associate a unique basis of (extended) $L^2$-harmonic forms $\mathcal{K}_{F}|F$ (see \cite{Atiyah}) on the fibre $F=X_u$. Then since $\sigma_{\Lambda'}(t)$ as defined in (\ref{action35}) is a smooth family of linear mappings it also acts smoothly in $t$ on $\mathcal{K}_{F}|F$ (where we used the notation $\mathcal{K}_F$ for the fibrewise $L^2$-harmonic forms as described in Appendix A). Finally applying Corollary \ref{harmonicbla} (substituting $\tilde X$ by $X_0$) extends $\overline \sigma_{\Lambda'}$ smoothly to $\mathcal{K}_{X_0}$.\\
Defining $P_{\pm}=\frac{Id\mp i\gamma}{\sqrt{2}}$ as the projection onto the $\pm i$ eigenspace of $\gamma$ we can shortly write $P_2(t), t \in [0,1]$ as the path of projections onto (note that $\tilde W=(P_+ +P_-)\tilde W$)
\begin{equation}\label{strlagr2}
L_{X_0, \Lambda'}^\infty(t) = \left(P_+\mathcal{L}\overline \sigma_{\Lambda'}(t)^* \mathcal{L}^{-1} + P_- \mathcal{L} \overline \sigma_{\Lambda'}(t)  \mathcal{L}^{-1}\right)(\tilde W)  \oplus \tilde W' \oplus F_\nu^-, \quad t \in [0,1],
\end{equation}
where $\overline \sigma_\alpha^*$ here again means taking the adjoint. 
\begin{Def}\label{strlagr0}
Let $P_{\mathcal{K}_{X_0}}: L^2(\Omega^*_{\partial X_0})\rightarrow \mathcal{K}_{X_0}$ be the orthogonal projection onto the closed subspace $\mathcal{K}_{X_0}\subset L^2(\Omega^*_{\partial X_0})$. Let $Q_{\Lambda'}(t), \ t \in [0,1]$ be the orthogonal projection onto the family of closed subspaces 
\begin{equation}\label{strlagr00}
L_{X_0, \Lambda'}(t) = \left(P_+\mathcal{L}\overline \sigma_{\Lambda'}(t)^* \mathcal{L}^{-1} + P_- \mathcal{L} \overline \sigma_{\Lambda'}(t)  \mathcal{L}^{-1}\right)P_{\mathcal{K}_{X_0}}(L_{X_0}) \oplus (I-P_{\mathcal{K}_{X_0}})L_{X_0}, 
\end{equation}
\end{Def}
{\it Remark.\ } The definition of $P_{\Lambda'}(t), \ \Lambda'\subset \Lambda$ means first stretching the collar of $X_0$ to infinity and taking the associated path of Calderon projectors of $\tilde X$, then applying the circle action to the adiabatic limit $L^\infty_{X_0}$ and finally running backwards from $L_{X_0}^\infty$ to $L_{X_0}$. That (\ref{action}) resp. (\ref{strlagr2}) and (\ref{strlagr00}) define paths of Lagrangians is a claim which is to be proved.\\
Note that the monomial basis $\{z^\alpha\}_{\alpha \in \Lambda}$ of $M(p)$ can always be chosen so that it contains $z^0=1 \in \mathbb{C}[z_0,\dots,z_n]$.
For the following note also that the action $\overline\sigma_\alpha$ on $H^*(X_0,\mathbb{C})$ evaluated for $t=1/\beta$ equals
\begin{equation}\label{monodromybla}
\overline\sigma_\alpha(1/\beta) =id \otimes (\rho_*\circ P_\alpha+(I-P_\alpha)):H^*(S^1,\mathbb{C})\otimes H^*(F,\mathbb{C})\rightarrow H^*(S^1,\mathbb{C})\otimes H^*(F,\mathbb{C})
\end{equation}
where $\rho_*$ is the algebraic monodromy of $X$, acting on $H^*(F,\mathbb{C})$ and $P_\alpha$ is the orthogonal projection onto the subspace corresponding to $\alpha \in \Lambda$. Using the isomorphism between $\mathcal{K}_{X_0}$ and $\tilde W$ as indicated above, we finally prove (\ref{actionmatrix}). Let $\mathcal{W}=\tilde W\oplus\gamma \tilde W$ be the symplectic subspace in $L^2(\Omega^*_{\partial X_0})$ generated by $\tilde W$. Let $M_{\Lambda'}$ be the diagonal matrix having as entries the (real) eigenvalues of the Gauss-Manin-connection $\nabla^{GM}$ acting on $\mathcal{H}^n_{X/D}$ with respect to the $\mathcal{O}_{D,0}$-basis of $\mathcal{H}''_{0} \simeq M(f)$ given by a set of $|\Lambda'|$ monomials $z^{\alpha(1)},\dots, z^{\alpha(\mu)}$ corresponding to $\Lambda'\subset \Lambda$ ($M_{\Lambda'}$ being zero on the diagonal entries corresponding to $\Lambda\setminus \Lambda'$). Define by $\tilde M_{\Lambda'}$ the matrix given by $M_{\Lambda'}\otimes M_{\Lambda'}$ with $M_{\Lambda'}$ acting on $\tilde W$ by replacing $\sigma_\alpha$ in (\ref{act}) and extended to the $\gamma \tilde W$-summand wrt of $\mathcal{W}$ in the obvious way. Here we consider $H^n(F,\mathbb{C})\subset \mathcal{K}_{X_0}$ following (\ref{act}) and (\ref{isom}) and $M_{\Lambda'}$ acts on $\tilde W$ relative to the images under $r$ of the basis elements determined by $z^{\alpha(1)},\dots, z^{\alpha(\mu)}$ in $H^n(F, \mathbb{C})\subset \mathcal{K}_{X_0}$. We define a path of Lagrangian subspaces in $\mathcal{W}\subset L^2(E|\partial X_0)$ by considering for $t \in [0,1]$ the image of 
\begin{equation}\label{actionmatrix2}
\tilde \sigma(t)(s):=e^{2\pi i t\gamma (\tilde M_{\Lambda'})}s, \quad  s \in \tilde W,\ t \in [0,1].
\end{equation}
where $\tilde M_{\Lambda'}$ acts on the image of $H^n(F,\mathbb{C})$ in $\tilde W$ as explained above and is extended to the $\gamma\tilde W$-summand of $\mathcal{W}$ by mapping the above basis elements of $H^n(F, \mathbb{C})\subset \tilde W$ to $\gamma \tilde W$ using $\gamma$. We then have
\begin{lemma}\label{actionmat}
With the above notation, we have that the path of Lagrangian subspaces $L_{X_0, \Lambda'}^\infty(t)$ as defined in (\ref{strlagr2}) equals
\[
L_{X_0, \Lambda'}^\infty(t) = \tilde \sigma(\tilde W)  \oplus \tilde W' \oplus F_\nu^-, \quad t \in [0,1].
\]
where $\tilde \sigma:[0,1]\times \tilde W\rightarrow \mathcal{W}$ is given as in (\ref{actionmatrix2}).
\end{lemma}
\begin{proof}
The assertion is an elementary calculation, namely write shortly $\hat M(t)=2\pi t\tilde M_{\Lambda'}$, then since $\hat M(t)$ and $\gamma$ commute for any $t \in [0,1]$
\[
\begin{split}
e^{i\gamma \hat M(t)}|\tilde W&=Id_{\tilde W}+ i\gamma \hat M(t)+ \frac{(i\gamma\hat M(t))^2}{2}+ \frac{(i\gamma\hat M(t))^3}{3!}+\frac{(i\gamma\hat M(t))^4}{4!} \dots\\
&=Id_{\tilde W}+ i\gamma \hat M(t)+ \frac{\hat M(t)^2}{2}+ \frac{i\gamma\hat M(t)^3}{3!}+\frac{\hat M(t)^4}{4!}\dots\\
&=\frac{1}{2}(Id_{\tilde W}+i\gamma)e^{i \hat M(t)}+\frac{1}{2}(Id_{\tilde W}-i\gamma)e^{-i \hat M(t)},
\end{split}
\]
which already gives the assertion inspecting (\ref{strlagr2}).
\end{proof}
As above we will set for any subset $\Lambda'\subset\Lambda$ (note again that by Lemma \ref{mult} below and Lemma 3.19 in \cite{klein2}, the $\overline\sigma_{\alpha},\ \alpha \in \Lambda$ commute pairwise)
\[
\overline \sigma_{\Lambda'}:=\prod_{\alpha\in \Lambda'}\overline\sigma_{\alpha}
\]
and then set $P_{\rho_*(L_{X_0}^\infty),\Lambda'}:=P_{L^\infty_{X_0,\Lambda'}(1/\beta)}$. We define a modified fractional part for $x \in \mathbb{R}$ as
\begin{equation}\label{fractionalpart3}
\{x\}':=\left\{\begin{matrix}\{x\} &{\rm if}\ 0\leq \{x\}\leq \frac{1}{2}\\ \{x\}-1&\frac{1}{2}<\{x\}<1\end{matrix}\right.,
\end{equation}
where $\{\cdot\}$ denotes the usual fractional part.
\begin{theorem}\label{theorem2}
For each $\Lambda'\subset \Lambda$ and any $t \in [0,1]$, $Q_{\Lambda'}(t) \in {\rm Gr}(A)$, $P_{\Lambda'}(t) \in {\rm Gr}_\infty(A)$ and $Q_{\Lambda'}(t),\ P_{\Lambda'}(t)$ are homotopic relative fixed endpoints. Furthermore, the associated family $\{{\rm SF}(D_{P_{\Lambda'}(t)})_{t \in [0,1]} \}_{\Lambda'\subset \Lambda}$ of spectral flows satisfies:
\begin{enumerate}
\item For each $\Lambda' \subset \Lambda$
\begin{equation}
{\rm SF}(D_{P_{\Lambda'}(t)})_{t \in [-1,2]} \ = -2\beta\sum_{\alpha \in \Lambda'}({\rm deg}(z^{\alpha}) +\sum^\mu_{i=1} w_i - 1).
\end{equation}
\item 
For any $\Lambda' \subset \Lambda$,
\begin{equation}
\begin{split}
\tilde \eta(D_{P_{\rho_*(L_{X_0}^\infty),\Lambda'}})-\tilde \eta(D_{P_{L_{X_0}^\infty}})=&\sum_ {\alpha \in \Lambda'}\{2({\rm deg}(z^{\alpha}) +\sum^\mu_{i=1} w_i)\}' \\
&+ \tau_\mu(P_{L_ {X_0}}, P_{L_{X_0}^\infty}, P_{\rho^*(L_{X_0}^\infty),\Lambda'})
\end{split}
\end{equation}
where $\tau_\mu$ denotes the triple index as introduced in Definition \ref{tripleindex} and $\{\cdot\}$ denotes the fractional part.
\item For $\alpha \in \Lambda$, the set of numbers $\{{\rm sf}(\alpha) := -\frac{1}{2}\rm{SF}(D_{P_{\alpha}(t)})_{t \in [0,1]}\}_{\alpha \in \Lambda} \in \mathbb{Z}$ and $\beta$ determine the 'variation structure' resp. the Seifert form of the quasihomogeneous hypersurface singularity given by $f$, more precisely we have
\[
\mathcal{V}(f)=\bigoplus_{\alpha \in \Lambda}\mathcal{W}_{{\rm exp}\left(2\pi i \frac{{\rm sf}(\alpha)}{\beta}\right)}((-1)^{[\frac{{\rm sf}(\alpha)}{\beta}]+n}),
\]
where $[\cdot]$ denotes the integral part, here we have used the notation for the eigenspace decomposition of the 'variation structure' introduced by Nemethi (\cite{nem1,nem2}) resp. in \cite{klein2} (Appendix A, Definition 4.9).
\end{enumerate}
\end{theorem}
We note that from the above discussion we have:
\begin{folg}\label{seidelfolg}
For $n \geq 2$ the symplectic monodromy $f$ defines
an element of infinite order in $\pi_0(Aut(M,\partial M,\omega))$ if $\rm{SF}(D_{P_{\alpha=0}(t)})_{t \in [-1,2]}$ is not equal to zero.
\end{folg}
{\it Remark. \ } We adopt the term 'variation structure' introduced by Nemethi \cite{nem1,nem2}, compare Definition 4.9 in \cite{klein2}. In short, it encodes the topological data given by the set $(U, b, h, V)$, where $U$ is the middle cohomology of the Milnor fibre, $b$ its intersection form, $h$ its monodromy and $V$ its variation mapping.\\
We can give an alternative (in a sense, more direct) interpretation of the above by considering the bundle $\tilde Z \rightarrow \tilde X \rightarrow S^1$ which is the $\beta$-fold cycling covering $\pi:\tilde X \rightarrow X$ of the Milnor bundle $Z \rightarrow X\rightarrow S^1$, that is, we have the commuting diagram:
\begin{equation}\label{betacov}
\begin{CD}
\tilde X @>>\pi> X \\
 @VV \tilde f V    @VV f V \\
S^1  @>>\lambda_\beta > S^1,
\end{CD}
\end{equation}
where $\lambda_\beta=z^\beta, z\in S^1$. Note that $\tilde X$ is diffeomorphic to the link of the polynomial $f(z_0,\dots,z_n)-z^\beta_{n+1}$ on $\mathbb{C}^{n+2}$. By the Wang exact sequence (set $Z_u =F$ for some $u \in S_\delta^1$)
\begin{equation}\label{wang}
0 \rightarrow H^n(\tilde X, \mathbb{C}) \xrightarrow {{\rm restr}}H^n (F,\mathbb{C})\xrightarrow {h^\beta -id}H^n(F,\mathbb{C}) \rightarrow H^{n+1}(\tilde X,\mathbb{C})\rightarrow 0,
\end{equation}
and since for a quasihomogeneous polynomial of weighted degree $\beta$ $h^\beta=id$ one has $H^*(\tilde X,\mathbb{C})\simeq H^*(S^1,\mathbb{C})\otimes H^*(F,\mathbb{C})$. Let now $\tilde g$ be the metric on $\tilde X$ constructed as follows. Note that $T\tilde f= \{{\rm ker } \tilde f_*: TY \rightarrow TS^1_\delta\}$ carries a canonical metric $g^{T\tilde f}$ induced by $\mathbb{C}^{n+1}$, on the other hand consider the 'Euler vector field' on $\mathbb{C}^{n+1}$,
\[
X_f(z)=\sum_{i=0}^n 2\pi i w_i z_i\frac{\partial}{\partial z_i}, z \in Y \quad {\rm it\ satisfies}\quad (X_f.f)(z)=2\pi i f(z),
\]
lifts to $\tilde X$ as $\tilde X_f$ and thus defines a horizontal distribution $H_{\tilde X_f} \subset T\tilde X$. Using this one can define a metric $\tilde g$ on $\tilde X$ as
\begin{equation}\label{metricbla}
\tilde g=g^{T\tilde f} \oplus \tilde f^*g^{S^1_\delta} \quad {\rm s.t.\  the \ splitting}\quad T\tilde X=T\tilde f \oplus H_{\tilde X_f} \quad{\rm is \ orthogonal},
\end{equation}
where as above $g^{T\tilde f}=g^{\mathbb{C}^{n+1}}|_{Tf}$, while $g^{S^1_\delta}$ is the standard metric on $S^1_\delta$. With these definitions, $\tilde f:Y\rightarrow S^1_\delta$ becomes a Riemannian submersion. Note also that $L_{\tilde X_f}\tilde g=0$, i.e. the fibres of $(\tilde X, \tilde g)$ are totally geodesic by \cite{vilms}. Note that we will assume in the following that $\tilde g$ is perturbed in a neighbourhood of the boundary as in Section 3 (Lemma 3.3) of \cite{klein2} to become canonically metrically trivial. Thus because of the splitting of the exterior derivative $d$ induced by $\tilde g$ on $\tilde X$ (as discussed in Appendix A), one can argue from
\[
\Omega^*(\tilde X,\mathbb{C})=\tilde f^*\Omega^*(S^1,\mathbb{C})\otimes \Omega^*_V(\tilde X,\mathbb{C})
\]
where $i_X^h\Omega^*_V(\tilde X,\mathbb{C})=0$ for any horizontal lift $X^h$ (resp. $\tilde g$) that
\begin{equation}\label{splitting}
H^*(\tilde X,\mathbb{C})\simeq\tilde f^*H^*(S^1,\mathbb{C})\otimes \Gamma_{||}({\bf H}^*(\tilde Z,\mathbb{C})),
\end{equation}
where ${\bf H}^*(\tilde Z,\mathbb{C})$ denotes the $\mathbb{Z}$-graded vector bundle whose fibre over $u \in S^1$ is the cohomology of the complex $\Omega^*(\tilde Z_u,\mathbb{C})$ (note $\tilde Z_u=F$) and this splitting carries over to the level of harmonic forms (Appendix A, see also below). $\Gamma_{||}$ indicates global parallel sections over $S^1$ w.r.t. to the connection $L_{X^h}$ ($X_h$ being any horizontal lift resp. $\tilde g$ of $X \in \mathcal{X}(S^1)$) acting on sections of $\Omega^*_V(\tilde X,\mathbb{C})$ (for more details on this connection, see \cite{bilo}).\\
Let now $\tilde D$ be the signature operator associated to $\tilde g$ on $\tilde X$. Consider again the 'space of extended $L^2$-solutions of $\tilde D\beta=0$' on $\tilde X_{\infty}$ in the sense of \cite{Atiyah} 
\[
\mathcal{K}_{\tilde X}:=\{\beta \in \Omega^{even}_{\tilde X}\  {\rm s.t.}\  \tilde D\beta=0\ {\rm and}\ r(\beta) \in F_0^-\oplus{\rm ker}(\tilde A)\},
\]
where $\tilde A$ denotes the tangential operator of $\tilde D$ over $\partial \tilde X$. For any $\alpha \in \Lambda$, we wish (as before) to define a circle action on $\mathcal{K}_{\tilde X}$. Recall there are  isomorphisms
\begin{equation}\label{L2kernel}
\begin{split}
\mathcal{K}_{\tilde X}&\simeq \left(im(H^{even}(\tilde X,\partial \tilde X,\mathbb{C})\rightarrow H^{even}(\tilde X,\mathbb{C}))\oplus im(H^*(\tilde X,\mathbb{C})\rightarrow H^*(\partial \tilde X, \mathbb{C}))\right)\\
&\simeq \sum_{\rm even } \tilde f^*\mathcal{H}^*(S^1,\mathbb{C}) \otimes \Gamma_0(\mathcal{K}_F),
\end{split}
\end{equation}
where $\Gamma_0(\mathcal{K}_F)$ denotes parallel sections in the bundle of fibrewise extended $L^2$-harmonic forms as described in Corollary \ref{harmonicbla} in Appendix A. So $\mathcal{K}_{\tilde X}$ can be identified (smoothly) with a subspace of $H^*(\tilde X,\mathbb{C})$, whereas by setting $\hat W = r(\mathcal{K}_{\tilde X})$ the Calderon projector on $\tilde X$ limits (by stretching the collar of $\tilde X$ to infinity) to the orthogonal decomposition
\begin{equation}\label{limitcomposition}
L^\infty_{\tilde X}= \hat W  \oplus \hat W' \oplus F_\nu^+,
\end{equation}
for some isotropic subspace $\hat W' \subset L^2(\Omega^*(\partial \tilde X,\mathbb{C}))$, $\nu\in \mathbb{N}$ is a number greater then the 'non-resonance level' of $\tilde X$. Then similar to (\ref{relpic}) associate to any $\alpha \in \Lambda$ a global section $\Phi_\alpha \in \Gamma({\bf H}^*(\tilde Z,\mathbb{C}))$. Now fixing any $u \in S^1_\delta$, denote by $\tau_u$ the isomorphism $\tau_u:H^*(\tilde Z_u,\mathbb{C}) \rightarrow \Gamma_{||}({\bf H}^*(\tilde Z,\mathbb{C}))$ associating to any element of the fibre cohomology of $\tilde Z_u$ the parallel section restricting to this element in $H^*(\tilde Z_u,\mathbb{C})$. Understanding this, we define using $\sigma:H^*(\tilde Z_u,\mathbb{C}))\rightarrow H^*(\tilde Z_u,\mathbb{C})$ as in \ref{circleact34} for any $\alpha \in  \Lambda$
\[
\begin{split}
&\sigma_{\alpha} :H^*(\tilde Z_u,\mathbb{C}) \rightarrow H^*(\tilde Z_u,\mathbb{C})\\
&\sigma_\alpha(t)=\sigma(t)\circ P_{\alpha}+(I-P_{\alpha}),
\end{split}
\]
where again $P_\alpha$ projects orthogonally onto the subspace spanned by $\phi_\alpha$.
Then set 
\begin{equation}\label{actioncov}
\begin{split}
 &\overline \sigma:[0,1] \times H^*(\tilde X,\mathbb{C}) \rightarrow H^*(\tilde X,\mathbb{C})\\
&\overline \sigma(t)=id \otimes \tau_u \circ \sigma_\alpha(t) \circ\tau_u^{-1},
\end{split}
\end{equation}
using the splitting (\ref{splitting}). Note that $\sigma$ is induced by the weighted circle action $\sigma_t(z)=(e^{2\pi i t\beta_0},\dots, e^{2\pi i t\beta_n})z$ on $\mathbb{C}^{n+1}$ pulled back to $\tilde X$ by $\pi$, acting on relative forms. So using (\ref{L2kernel}) we get a smooth $S^1$-action for each $\Lambda' \in \Lambda$:
\begin{equation}\label{harmonicaction}
\overline \sigma_{\Lambda'}:S^1 \times \mathcal{K}_{\tilde X} \rightarrow \mathcal{K}_{\tilde X}.
\end{equation}
We arrive at the following path of (proven to be) Lagrangians (note that $\mathcal{L}$ is the restriction map composed with some linear map as above s.t. $\mathcal{L}$ acts isometric)
\begin{equation}\label{strlagr3}
L_{\tilde X,\Lambda'}^\infty(t) := \left(P_ +\mathcal{L}\overline\sigma_{\Lambda'}(t)^* \mathcal{L}^{-1} + P_- \mathcal{L} \overline \sigma_{\Lambda'}(t)\mathcal{L}^{-1}\right)(\hat W)  \oplus \hat W' \oplus F_\nu^+, \quad t \in [0,1].
\end{equation}
Define furthermore a path $Q_{\Lambda'}(t), \ \Lambda\subset \Lambda'$, proven to be in ${\rm Gr}(A)$, analogously as in (\ref{strlagr00}), replacing $\mathcal{K}_{X_0}$ by $\mathcal{K}_{\tilde X}$. Note as above that the evaluation of $\overline\sigma_{\Lambda'}$ on $H^*(\tilde X,\mathbb{C})$ for $t=1/\beta$ equals $id \otimes (\rho_*\circ P_{\Lambda'}+(I-P_{\Lambda'}))$ on $H^*(S^1,\mathbb{C})\otimes H^*(F,\mathbb{C})\simeq H^*(\tilde X,\mathbb{C})$ where $\rho_*$ is the algebraic monodromy of $X$, acting on $H^*(F,\mathbb{C})$, so set $P_{\rho_*(L_{\tilde X}^\infty),\Lambda'}:=P_{L_{\tilde X,\Lambda'}^\infty(1/\beta)}$ analogously to above. For the following we set ${\rm deg}(z_i))=w_i$ which defines multiplicatively a 'weighted degree' on  each $z^\alpha,\ \alpha \in \Lambda$.
\begin{theorem}\label{theorem3}
Define for each $\Lambda' \in \Lambda$ a family of self-adjoint extensions $D_{P_{\Lambda'}(t)}$ of the signature operator $D$ on $\tilde X$ by composing the following three paths to a path $P_{\Lambda'}(t) \in {\rm Gr}_\infty(A), t \in [-1,2]$:
\[
\begin{split}
 1.\quad P_1(t)&=P_{L_{\tilde X}^{1/-t}}, \ t \in [-1,0],\\
 2.\quad P_2(t)&= P_{L_{\tilde X,\Lambda'}^\infty(t)},\ t \in [0,1],\\
 3.\quad P_3(t)&= P_{L_{\tilde X}^{1/(t-1)}}, t \in [1,2].
\end{split}
\]
Then the paths of projections $P_{\Lambda'}(t), Q_{\Lambda'}(t)\in {\rm Gr}(A)$ are homotopic relative fixed endpoints. The family $\{\rm{SF}(D_{P_{\Lambda'}(t)})_{t \in [-1,2]} \}_{\Lambda \in \Lambda'}$ of spectral flows satisfies
\begin{enumerate}
\item For each $\Lambda'\subset \Lambda$
\begin{equation}
{\rm SF}(D_{P_{\Lambda'}(t)})_{t \in [-1,2]} \ =  -2\beta\sum_{\alpha \in \Lambda'}({\rm deg}(z^{\alpha}) +\sum^\mu_{i=1} w_i - 1)
\end{equation}
\item 
For any $\Lambda' \subset \Lambda$,
\begin{equation} \label{etadiff34}
\begin{split}
\tilde \eta(D_{P_{\rho^*(L_{\tilde X}^\infty),\Lambda'}})-\tilde \eta(D_{P_{L_{\tilde X}^\infty}})=&\sum_{\alpha\in \Lambda'}\{2({\rm deg}(z^{\alpha}) +\sum^\mu_{i=1} w_i)\}'\\
&+\tau_\mu(P_{L_ {\tilde X}}, P_{L_{\tilde X}^\infty}, P_{\rho^*(L_{\tilde X}^\infty),\Lambda'})
\end{split}
\end{equation}
where again, $\tau_\mu$ denotes the triple index and $\{\cdot\}'$ is the fractional part defined in (\ref{fractionalpart3}).
\item The set of numbers $\{{\rm sf}(\alpha) := -\frac{1}{2}{\rm SF}(D_{P_{\alpha}(t)})_{t \in [0,1]}\}_{\alpha \in \Lambda} \in \mathbb{Z}$ and $\beta$ determine the 'variation structure' resp. the Seifert form of the quasihomogeneous hypersurface singularity given by $p$, more precisely we have
\[
\mathcal{V}(f)=\bigoplus_{\alpha \in \Lambda}\mathcal{W}_{exp(2\pi i \frac{{\rm sf}(\alpha)}{\beta})}((-1)^{[\frac{{\rm sf}(\alpha)}{\beta}]+n}),
\]
where $[\cdot]$ denotes the integral part, here we have used the notation for the eigenspace decomposition of the 'variation structure' introduced by Nemethi (\cite{nem1,nem2}).
\end{enumerate}
\end{theorem}
\begin{proof}
We will first focus on the proof of Theorem \ref{theorem2}, the proof of Theorem \ref{theorem3} is similar and will be focussed afterwards, note that we will in the following frequently suppress the indices $\Lambda'$. We have the following decomposition into symplectic subspaces 
\[
L^2(\Omega^*_{\partial X})=(F^-_\nu\oplus F_\nu^+) \oplus(d(E_\nu^+)\oplus d^*(E_\nu^-))\oplus (d^*(E_\nu^+)\oplus d(E_\nu^-))\oplus \rm{ker}\ A.
\]
With respect to this decomposition the adiabatic limit of the Cauchy data space of $X$ decomposes as follows (cf. Theorem 2.16 in \cite{klein2}):
\begin{equation}\label{decomp}
\lim_{r \rightarrow \infty}L^r_X=F_\nu^+ \oplus (W_X\oplus\gamma(W_X^\perp))\oplus d(E^-_\nu)\oplus V_X.
\end{equation}
Here, in the above terminology, $W= W_X \oplus V_X \subset V_\nu$, i.e. $W_X \subset d(E_\nu^+)$ is isomorphic to ${\rm im}(H^{even}(X,\partial X,\mathbb{C})\rightarrow H^{even}(X,\mathbb{C}))$, $V_X\subset {\rm  ker\ }A$ is the symplectic reduction of the Cauchy data space of $X$ with respect to $F_0^-$:
\[
V_X=R_0(L_X)=\frac{L_X\cap (F_0^-\oplus {\rm ker \ } A)}{L_X\cap F_0^-} \subset {\rm ker\ }A,
\]
one has $V_X \simeq im(H^*(X,\mathbb{C})\rightarrow H^*(\partial X, \mathbb{C}))$ (see Theorem 2.16 in \cite{klein2}). Analogously with $W_{X_0}\subset d(E^-_\nu)$
\begin{equation}\label{decomp2}
\lim_{r \rightarrow \infty}L^r_{X_0}=F_\nu^- \oplus d(E^+_\nu)\oplus (\gamma(W_{X_0}^\perp) \oplus W_{X_0})\oplus V_{X_0}.
\end{equation}
This formula already shows using $\mathcal{K}_{X^r_0}\simeq W_{X_0}\oplus V_{X_0}=\tilde W$ for any $r>0$, using $L_{X_0}^r=P_{\mathcal{K}_{X^r_0}}(L_{X^r_0})\oplus (I-P_{\mathcal{K}_{X^r_0}})(L_{X^r_0})$, that $Q(t)$ and $P(t)$ are homotopic relative fixed endpoints as projections, that this homotopy is a homotopy in ${\rm Gr}(A)$  will follow from the proof of Lemma \ref{Lagrangian} below.\\
Using Nicolaescu's Theorem \ref{nicolaescu} we have 
\[
{\rm SF}(D_{P(t)})_{t\in [0,3]}={\rm Mas}({\rm ker} P(t), L_X(t))_{t \in [0,3]}.
\]
Now $L_X(t)=L_X$ is homotopic relative endpoints to the composite of three paths, the first stretches $L_X$ to its adiabatic limit $L_X^\infty$, the second is the constant path at $L_X^\infty$ and the third traces the first path backwards to $L_X$. Using homotopy invariance, additivity of the Maslov index and the definition of $P(t)$ as a composition of three paths, we can thus write ${\rm Mas}({\rm ker} P(t), L_X(t))_{t \in [0,3]}$ as a sum of three terms $M_{1,2,3}$, more explicitly:
\[
\begin{split}
M_1&={\rm Mas}(L_{X_0}^{1/(1-t)}, L_X^{1/(-t)})_{t \in [-1,0]},\\
M_2&={\rm Mas}(\hat \sigma(t-1) (L_{X_0}^\infty),L_X^\infty)_{t \in [0,1]},\\
M_3&={\rm Mas}(L_{X_0}^{1/(t-2)}, L_{X}^{1/(t-1)})_{t \in [1,2]}.
\end{split}
\]
The following lemma gives the vanishing of $M_1$ and $M_3$.
\begin{lemma}
The dimension of the intersection $L^r_X\cap L^r_{X_0}$ is independent of $r \in [0,\infty]$.
\end{lemma} 
\begin{proof}
For all $r  < \infty$ the intersection $L^r_X \cap L^{r}_{X_0}$ is isomorphic to the kernel of $D$ acting on the closed manifold $X_{r}=X_r \cup X_{0,r}$ (the index $r$ means the elongation of the respective manifold by a metric cylinder of length $r$ glued to its boundary), which is an homotopy invariant isomorphic to $H^*(X\cup X_0,\mathbb{C})$, in particular its dimension is independent of $r$. On the other hand, $L^\infty_X\cap L^\infty_{X_0}=W_X \oplus W_{X_0}\oplus (V_X\cap V_{X_0})$ by the above decomposition, the latter can be shown (see \cite{lesch}) to be isomorphic to $H^*(X\cup X_0,\mathbb{C})$, so the dimension of the intersection is constant for $r\in [0,\infty]$.
\end{proof}
Summarizing, the computation of the spectral flow reduces to
\[
{\rm SF}(D_{P(t)})_{t\in [0,3]}={\rm Mas}({\rm ker}\ P_2(t), L^\infty_X)_{t \in [0,1]},
\]
and due to the decompositition (\ref{decomp2}) of ${\rm ker\ } P_2(t)$ and  $L^\infty_X$ this reduces to a calculation in finite dimensions, as will follow.\\
For $\alpha(j) \in \Lambda\ , j\in \{1,\dots,\mu\}$ let $z^{\alpha(1)},\dots,z^{\alpha(\mu)}$ represent global sections $\Phi_1,\dots,\Phi_\mu$ of $f_*\Omega^{n+1}_{X}/(df \wedge d(f_*\Omega^{n-1}_{X})$, 
which restrict to a basis $\omega_1,\dots,\omega_\mu$ of $H^n(F,\mathbb{C})$ in at least $u \in
S^1_\delta$ (see \cite{loo}, Appendix A of \cite{klein2}), so $df\wedge \phi_i =z^{\alpha(i)} dz_0\wedge \dots\wedge dz_n$. 
Consider now the vector field $K= \sum_i w_i z_i \frac{\partial}{\partial x}$ on $\mathbb{C}^{n+1}$, this defines a horizontal lift of the standard vector field $\partial / \partial u$ on
$D_\delta$. Set $deg(z_i)=w_i$.
\begin{lemma}\label{mult}
With the above notation we have
\[
\mathcal{L}_K \Phi_j = ({\rm deg}(z^{\alpha(j)}) +\sum^\mu_{i=1} w_i - 1)\phi_j =:d_j/\beta\Phi_j.
\]
where $j \in\ \{1,\dots, \mu\}$.
\end{lemma}
\begin{proof}
We have for $j \in\{1,\dots,\mu\}$
\[
\mathcal{L}_{K}(\Phi_j) =
(i_Kd+di_K )(\Phi_j) \ {\rm mod}(df_*\Omega_{X/S}^{n-1}))=(i_Kd)(\Phi_j) \ {\rm
mod}(df_*\Omega_{X/S}^{n-1}).
\]
Using the  isomorphism $df \wedge \dots:
f_*\Omega^n_{X/S}/df_*\Omega^{n-1}_{X/S} \simeq f_*\Omega^{n+1}_X/df \wedge
f_*d\Omega^{n-1}_X$ we have as mentioned above $\phi_j \simeq \alpha_J dz_0
\wedge \dots \wedge z_n$ with $\alpha_j=z_0^{i_0} \dots  z_n^{i_n}$ a monomial.
Now since (note that $i_{K}df=f$)
\[
\begin{split}
 \mathcal{L}_{K}(df \wedge \Phi_i) \ =&\ \mathcal{L}_{K}df\wedge \Phi_i +
df\wedge \mathcal{L}_{K} \Phi_i\\
=&\ df\wedge \tilde \Phi_i + df \wedge i_{K} d\tilde \phi_i \quad mod (df \wedge
f_*d\Omega^{n-1}_X).
\end{split}
\]
we have $\mathcal{L}_{K}(df\wedge\Phi_j) - df\wedge \Phi_j \simeq
\mathcal{L}_{K}\tilde \Phi_j$. Furthermore denote by
$\Phi_t(z)=(e^{w_0}t,\dots,e^{w_n}t)$ the flow of $K$ on $\mathbb{C}^{n+1}$,
then
\[
\begin{split}
\mathcal{L}_{K}(df\wedge \Phi_j) \ =&\ \mathcal{L}_{K}(\alpha_j
dz_0\wedge\dots\wedge dz_n)\\
=&\ \frac{d}{dt}(\Phi_t^*)|_{t=0}\alpha_j dz_0\wedge\dots\wedge dz_n\\
=&\ \frac{d}{dt}|_{t=0}\prod_k e^{w_k i_kt}\prod_i e^{w_it}\alpha_j
dz_0\wedge\dots \wedge dz_n\\
=&\ i(\sum_k w_k i_k+\sum_k w_k)\alpha_j dz_0\wedge\dots\wedge dz_n.
\end{split}
 \]
Putting this together we arrive at
\[
\mathcal{L}_{K}\Phi_j= (\sum_k w_k(i_k+1) - 1)\Phi_j,
\]
hence the desired formula.
\end{proof}
We can now give the proof of Lemma \ref{par}:
\begin{proof}
Wit the notation from the previous proof, the parallel global section of $\mathcal{H}^n(f_*\Omega^\cdot_{X/D^*_\delta})|_{S^1_{|u|}}$ which restricts to $\phi_j(u)$ for a fixed $u \in S_{|u|}^1$ (corresponding to $\alpha_j\in\Lambda$) in $X_u$ is given for $t \in S^1$ by $t \mapsto t^{-d_j}\Phi_j(ut)$, so for $t=e^{2\pi i\vartheta},\vartheta \in [0,1]$
\[
\begin{split}
\mathcal{P}_{\vartheta}^{-1}\circ \tilde \sigma_{\alpha_j}|_{X_t} \circ \mathcal{P}_{\vartheta}(\phi_i)
&=\mathcal{P}_{\vartheta}^{-1}\tilde \sigma_{\alpha_j}(t^{-d_j}\Phi_j(ut)\\
&= \mathcal{P}_{\vartheta}^{-1}\Phi_j(ut)\\
&= t^{d_j}\Phi_j(u)\
\end{split}
\]
On the other hand by definition,
\[
\sigma_{\alpha_j}(\vartheta)(\phi_j)=e^{2\pi i \vartheta d_j}\phi_j=t^{d_j}\phi_j,
\]
which gives the assertion.
\end{proof}
Note that the 'Euler vector field' $\beta K$ generates the weighted circle action $\sigma$ on $\mathbb{C}^{n+1}$. Using this it is now easy to prove that the above defined path ${\rm im}\ P_2(t)$ is in fact Lagrangian:
\begin{lemma} \label{Lagrangian}
For any subset $\Lambda'\subset \Lambda$ the path ${\rm im}(P(t))$ of subspaces in $L^2(\Omega^*(\partial X_0,\mathbb{C})$ associated to $\sigma=\prod_{\alpha \in \Lambda'}\sigma_\alpha$ as given by (\ref{action}) (equivalently (\ref{strlagr2}) and the path associated to ${\rm im}(Q(t))$ as in (\ref{strlagr00}) are Lagrangian, more precisely, $P(t) \in {\rm Gr}_\infty(A)$, $Q(t) \in {\rm Gr}(A)$.
\end{lemma}
\begin{proof}
Starting from the decomposition (setting as above $\tilde W=W_{X_0}\oplus V_{X_0}$)
\[
L^\infty_{X_0}=\Lambda_{X_0} \oplus F_\nu^+= (\tilde W  \oplus \tilde W') \oplus F_\nu^-,
\]
we defined (suppressing $\Lambda'$)
\[ 
L^\infty_{X_0,t}=\left(P_ +\mathcal{L}\overline \sigma^*(t)\mathcal{L}^{-1} + P_- \mathcal{L} \overline \sigma(t)  \mathcal{L}^{-1}\right)(\tilde W)  \oplus \tilde W' \oplus F_\nu^-, \quad t \in [0,1].
\]
We will show that with respect to a special choice of basis ${\rm span}\ \{e_1,\dots,e_l\}=\tilde W$ $L^\infty_{X_0,t}\cap V_\nu$ can be written as the graph of an isometry $\phi(\Lambda_{X_0,t}): ker(\gamma -i)\cap V_\nu \rightarrow {\rm ker}(\gamma+i)\cap V_\nu$ for all $t \in [0,1]$. Set $k={\rm dim\ } im(H^*(X_0,\mathbb{C})\rightarrow H^*(\partial X_0, \mathbb{C}))$, by the long exact sequence (set $F:= X_u$)
\begin{equation}\label{seq21}
0 \rightarrow H^{n-1}(\partial F,\mathbb{C}) \xrightarrow {\delta} H^n(F,\partial F,\mathbb{C}) \xrightarrow {j} H^{n}(F,\mathbb{C}) \xrightarrow {r} H_{n}(\partial F,\mathbb{C})\rightarrow 0.
\end{equation}
we have $H^n(F,\mathbb{C})\simeq {\rm im}\ (j:H^n(F,\mathbb{C})\rightarrow H^n(\partial F, \mathbb{C})) \oplus {\rm im}\ (r:H^n(F,\partial F,\mathbb{C})\rightarrow H^n(F,\mathbb{C}))$, this remains true by replacing $F$ by $X_0$, so in fact $r^{-1}(\tilde W) \subset H^*(X_0,\mathbb{C})$ using (\ref{isom}), shortly $\tilde W=V_{X_0}\oplus W_{X_0}$, so ${\rm dim} \ W_{X_0}=\mu-k$, ${\rm dim} \ V_{X_0}=k$. So writing
\[
H^n(F,\mathbb{C})\otimes H^*(S^1,\mathbb{C})={\rm span}\{\phi_i\otimes e_j\}_{i\in\{1,\dots,\mu\}, j\in\{0,1\}},
\]
where the $\{\phi_j\} \subset H^n(F,\mathbb{C})$ are associated to the set $\{\alpha_j\}\subset \Lambda$ as above. Following Lemma 3.19 in \cite{klein2}, the monomials can be chosen to diagonalize the $L^2$-innerproduct. So after appropriate (re)ordering
\[
\begin{split}
r^{-1}(W_{X_0})&={\rm span}\ \{\phi_1\otimes e_I,\dots,\phi_{\mu-k}\otimes e_I\},\\
r^{-1}(V_{X_0})&={\rm span}\ \{\phi_i \otimes e_I\}_{i\in\{\mu-k+1,\dots,\mu\}},
\end{split}
\]
where $I \in \{0,1\}$ depending on $n$ even or odd. Hence
\begin{equation}\label{orthbasis}
\begin{split}
W_{X_0}&={\rm span}\ \{\kappa(i^*\phi_j\otimes e_I+i^*(*\phi_j\otimes *e_I)):=f_{j}\}_{j\in\{1,\dots,\mu-k\}},\\
V_{X_0}&={\rm span}\ \{\kappa(i^*\phi_i \otimes e_I+i^*(*\phi_i \otimes *e_I)):=g_{i}\}_{i\in\{\mu-k+1,\dots,\mu\}},
\end{split}
\end{equation}
and the sets $\{f_j\}_j$ and $\{g_{j}\}_{j}$ are orthogonal resp. the inner product of $L^2(\Omega^*(\partial X_0)$.
So for $\alpha_i\in \Lambda$
\[
\begin{split}
(P_ +\mathcal{L}\sigma_{\alpha_i}(t)^* \mathcal{L}^{-1} + P_- \mathcal{L} \sigma_{\alpha_i}(t) \mathcal{L}^{-1})(f_i)\\
=P_+e^{-2\pi itd_j}\delta_{ij}f_j+P_-e^{2\pi itd_j}\delta_{ij}f_j,
\end{split}
\]
using the previous Lemma. Hence 
\[
\phi(\Lambda_{X_0,t,\alpha_j})P_+(f_i)=\frac{e^{2\pi itd_j}\delta_{ij}}{e^{-2\pi itd_j}}P_-(f_j)=\begin{pmatrix}e^{4\pi i t d_j},\ i=j\\ 0, \ i\neq j\end{pmatrix}P_-(f_j),
\]
analogously for the $\{g_i\}_i$. Consequently, the matrix $\phi(\Lambda_{X_0,t,\alpha_i})$ is unitary for all $t \in [0,1]$, which is the assertion for the path $P(t)$. The procedure for $L_{X_0, \Lambda'}(t)={\rm im}(Q(t))$ is the same, substituting the above decomposition of $L^\infty_{X_0}$ by
\[
L_{X^r_0}=P_{\mathcal{K}_{X^r_0}}(L_{X^r_0})\oplus (I-P_{\mathcal{K}_{X^r_0}})(L_{X^r_0}),
\]
and introducing  the definition (\ref{strlagr00}) for any $r>0$, this will also prove that the homotopy between and $Q(t)$ and $P(t)$ themselves are Lagrangian.
Finally, for any $t \in [0,1]$, $P(t) \in {\rm Gr}_\infty(A)$, since its image differs from $L_{X^\infty_0}\in {\rm Gr}_\infty(A)$ by a finite-dimensional subspace of smooth sections, furthermore, $Q(t)$ is in ${\rm Gr}(A)$ since it differs from $P_{X_0}$ by a finite, hence compact projection.
\end{proof}
\begin{lemma}
With the above notations,
\[
{\rm Mas}({\rm ker}\ P_2(t), L^\infty_X)_{t \in [0,1]}=-2\beta ({\rm deg}(\alpha(j)) +\sum^\mu_{i=1} w_i - 1).
\]
\end{lemma}
\begin{proof}
Consider the isomorphisms as defined above:
\[
\begin{split}
i_1 =(r\circ\kappa)^{-1}: \tilde W= W_{X_0} \oplus V_{X_0} &\simeq {\rm im}(H^{even}(X_0,\partial X_0,\mathbb{C})\rightarrow H^{even}(X_0,\mathbb{C}))\\
&\oplus im(H^*(X_0,\mathbb{C})\rightarrow H^*(\partial X_0, \mathbb{C}))\\
i_2: H^*(X_0,\mathbb{C})&\simeq  H^*(X_u,\mathbb{C})\otimes H^*(S^1,\mathbb{C}).
\end{split}
\]
Using the notations from the proof of Lemma \ref{Lagrangian}, $i_2\circ i_1:\tilde W \rightarrow H^*(X_u,\mathbb{C})\otimes H^*(S^1,\mathbb{C})$ maps the orthonormal basis $\{f_j\}_j \oplus \{g_i\}_i \subset W_{X_0}\oplus V_{X_0} = \tilde W \subset L^2(\Omega^*(\partial X_0,\mathbb{C}))$ to the basis $\{\phi_0,\dots,\phi_\mu\}\otimes \{\tilde e_I\}  \subset H^*(X_u,\mathbb{C})\otimes H^*(S^1,\mathbb{C})$ where the $\{\phi_j\}$ are associated to $\alpha_j\subset \Lambda$, which by Lemma 3.19 in \cite{klein2} can be chosen to be an orthonormal basis in the $L^2$ inner product on the fibre $X_u$, so
\[
i_2\circ i_1(f_j)=\phi_i \otimes \tilde e_I,\ i \in \{0,\dots,\mu-k \},\ \quad  i_2\circ i_1(g_j)=\phi_i \otimes \tilde e_I,\ i \in \{\mu-k+1,\dots,\mu\},\ I \in\{0,1\}.
\]
We now choose a special basis for $L^2(\Omega^*_{\partial X})=L^2(\Omega^*_{\partial X_0})$, recall
\begin{equation}\label{decomp123}
L^2(\Omega^*_{\partial X})=(F^-_\nu\oplus F_\nu^+) \oplus(d(E_\nu^+)\oplus
d^*(E_\nu^-))\oplus (d^*(E_\nu^+)\oplus d(E_\nu^-))\oplus \rm{ker}\ A,
\end{equation}
furthermore
\[
\begin{split}
L^\infty_X=F_\nu^+ \oplus (W_X\oplus\gamma(W_X^\perp))\oplus d(E^-_\nu)\oplus V_X,\\
L^\infty_{X_0}=F_\nu^- \oplus d(E^+_\nu)\oplus
(\gamma(W_{X_0}^\perp) \oplus W_{X_0})\oplus V_{X_0}.
\end{split}
\]
We choose an orthonormal basis adapted to $L^\infty_{X_0}$ as follows (we omit the contribution of $H^0(\partial X_u,\mathbb{C})\otimes H^*(S^1,\mathbb{C}) \subset V_{X_0}$ in the following):
\begin{equation}\label{basislag}
\begin{split}
F_\nu^+\oplus F_\nu^-&={\rm span}(e_i,\gamma e_i), \ i=\nu+1,\dots,\infty,\\
d(E_\nu^+)\oplus d^*(E_\nu^-)&={\rm span}(e_i,\gamma e_i),  \ i=l+1,\dots,\nu,\\
d^*(E_\nu^+)\oplus d(E_\nu^-)&=\gamma(W_{X_0} \oplus W_{X_0}^\perp)\oplus (W_{X_0} \oplus W_{X_0}^\perp)\\
&={\rm span}(\gamma(e_i^{W_{X_0}}, e_j^{W_{X_0}^\perp}),(e_i^{W_{X_0}}, e_j^{W_{X_0}^\perp}), \\
& i\in\{k+1,\dots,\mu\},j\in\{\mu+1,\dots,l\} \\
V_{X_0}\oplus \gamma V_{X_0}&= {\rm ker}(A)={\rm span}\{e^{V_{X_0}}_i,\gamma e^{V_{X_0}}_i\}, \ i\in \{1,\dots, k\}.
\end{split}
\end{equation}
Note here, that with the above notation $\{(e_i^{V_{X_0}}, e_j^{W_{X_0}})\}_{ij}=\{(g_i,f_j)\}_{ij}, \ j=j,\dots,\mu-k, i=\mu-k+1,\dots,\mu$. If, for an isotropic subspace $V\subset L^2(\omega^*(\partial X_0,\mathbb{C}))$, $P_V$ denotes the associated projection, i.e. ${\rm im}(P_V)=V$, whereas $\Phi(P_V)$ denotes the associated isometry $\Phi(P_V):P_-(V\oplus\gamma V)\rightarrow P_+(V\oplus \gamma V)$ , we have, referring to the decomposition (\ref{decomp123}) into symplectic subspaces:
\[
\begin{split}
\Phi(P_{L^\infty_X})=\Phi(P_{F_\nu^+}) \oplus
\Phi(P_{W_X\oplus\gamma(W_X^\perp)})\oplus \Phi(P_{d(E^-_\nu)})\oplus \Phi(P_{V_X}),\\
\Phi(P_{L^\infty_{X_0}})=\Phi(P_{F_\nu^-}) \oplus \Phi(P_{d(E^+_\nu)})\oplus
\Phi(P_{(\gamma(W_{X_0}^\perp) \oplus W_{X_0}})\oplus \Phi(P_{V_{X_0}}).
\end{split}
\]
Then, relative to the basis introduced in (\ref{basislag}), note that  if $\{e_i\}_i$ is any orthonormal basis for a Lagrangian $V$ in a symplectic space $(W,\gamma)$, then $\left\{P_\mp e_i:=1/\sqrt{2}(I\pm i\gamma)e_i\right\}_i$ spans ${\rm ker}\ (\gamma\pm i)\subset W$:
\begin{equation}\label{sigmat}
\begin{split}
\hat \sigma(t) \circ \Phi(P_{\gamma(W_{X_0}^\perp) \oplus W_{X_0}})\begin{pmatrix}P_-(\tilde W_{X_0}^\perp)\\ P_-(W_{X_0})\end{pmatrix}&=\begin{pmatrix}{\rm diag}(-i) & 0\\0& {\rm diag}(e^{2\pi i d_jt})_{j=k+1,\dots,\mu}\end{pmatrix}\begin{pmatrix}P_+(\tilde W_{X_0}^\perp)\\ P_+(W_{X_0})\end{pmatrix},\\
\hat \sigma(t) \circ \Phi(P_{V_{X_0}})(P_-(V_{X_0}))&=\left({\rm diag}(e^{4\pi i d_jt})_{j=1,\dots,k} \right)(P_+(V_{X_0})),\\
\end{split}
\end{equation}
Note that here, again, $k={\rm dim \ coker} (j)$, where $j:H^n(F,\partial F,\mathbb{C})\rightarrow H^n(F,\mathbb{C})$ is the canonical mapping, $\mu$ the Milnor number, while the $\{d_j\}_j$ were defined in Lemma \ref{mult}. Finally one calculates (note that $\Phi(I-P)=-\Phi(P)$):
\[
\begin{split}
{\rm Mas}&({\rm ker}\ P_2(t), L^\infty_X)_{t \in [1,2]}=-{\rm wind}(\Phi(I-P_2(t))\Phi^*(P_{L^\infty_X}))\\
&= -\frac{1}{2\pi i} \int_0^{2\pi} {\rm tr} \{(-\Phi(P_2(t))\Phi^*(P_{L^\infty_X}))^{-1}\\ 
&\cdot \frac{d}{dt}\left(-\left(\Phi(P_{F_\nu^-}) \oplus \Phi(P_{d(E^+_\nu)})\oplus \hat \sigma(t)\circ\left\{\Phi(P_{\gamma(W_{X_0}^\perp) \oplus W_{X_0}})\oplus \Phi(P_{V_{X_0}})\right\}\right)\Phi^*(P_{L^\infty_X})\right)\}dt
\end{split}
\]
The constant summands drop out, consequently
\[
\begin{split}
{\rm Mas}&({\rm ker}\ P_2(t), L^\infty_X)_{t \in [1,2]}\\
&= -\frac{1}{2\pi i} \int_0^{2\pi}{\rm tr} \{\left(-(\Phi(P_{d(E^-_\nu)})\oplus \Phi(P_{V_X}))
(\sigma^*(t)\circ \Phi^*(P_{\gamma(W_{X_0}^\perp) \oplus W_{X_0}})\oplus \sigma^*(t)\circ \Phi^*(P_{V_{X_0}}))\right) \\
&\cdot \frac{d}{dt}\left(-(\hat \sigma(t)\circ \Phi(P_{\gamma(W_{X_0}^\perp) \oplus W_{X_0}})\oplus \hat \sigma(t)\circ\Phi(P_{V_{X_0}}))(\Phi^*(P_{d(E^-_\nu)})\oplus \Phi^*(P_{V_{X}}))\right)\}dt.
\end{split}
\]
So after cyclic permutation under the trace
\[
\begin{split}
{\rm Mas}&({\rm ker}\ P_2(t), L^\infty_X)_{t \in [1,2]}\\
 &=-\frac{1}{2\pi i} \int_0^{2\pi} tr \{(\sigma^*(t)\circ \Phi^*(P_{\gamma(W_{X_0}^\perp) \oplus W_{X_0}})\oplus \sigma^*(t) \circ \Phi^*(P_{V_{X_0}})) \\
&\cdot \frac{d}{dt}\left(\hat \sigma(t)\circ \Phi(P_{\gamma(W_{X_0}^\perp) \oplus W_{X_0}}) \oplus \hat \sigma(t)\circ \Phi(P_{V_{X_0}})\right)\}dt\\
&= -(\sum_{i=1}^k 2d_i + \sum_{j=k+1}^\mu 2d_j ).
\end{split}
\]
where in the last step we used equation (\ref{sigmat}). Finally note that the monomials $z^{\alpha(i)}, \alpha(i) \in \Lambda $ can be viewed to be weighted homogeneous with weights $\tilde w_i=\beta_i/(\beta \cdot {\rm deg}(z^{\alpha(j)})$, since if $\alpha(j)=z^{i_0}\dots z^{i_n}$ then ${\rm deg}(\alpha(j))=i_0w_0+\dots+i_nw_n$ by definition.\\
To prove the second part of Theorem \ref{theorem2} we use Theorem 2.11 in \cite{klein2} to infer that since  $P_{\rho^*(L_{X_0}^\infty),\Lambda'}, P_{L_{X_0}^\infty}\in {\rm Gr}_\infty(A)$ we have the equalities (recall that $P_{L_{X_0}}$ denotes the Calderon projector of $X_0$):
\begin{equation}\label{etadiff45}
\begin{split}
&\tilde \eta(D_{P_{\rho^*(L_{X_0}^\infty),\Lambda'}})-\tilde \eta(D_{P_{L_{X_0}}})=\frac{1}{2\pi i} {\rm tr\ log}
(\Phi(P_{\rho^*(L_{X_0}^\infty),\Lambda'})\Phi(P_{L_{X_0}})^*)\\
&\tilde \eta(D_{P_{L_{X_0}^\infty}})-\tilde \eta(D_{P_{L_{X_0}}})=\frac{1}{2\pi i} {\rm tr\ log}
(\Phi(P_{L_{X_0}^\infty})\Phi(P_{ L_{X_0}})),
\end {split}
\end{equation}
so substracting gives:
\begin{equation}\label{trlog}
\begin{split}
&\tilde \eta(D_{P_{\rho^*(L_{X_0}^\infty),\Lambda'}})-\tilde \eta(D_{P_{L_{X_0}^\infty}})=\frac{1}{2\pi i} {\rm tr\ log}(\Phi(P_{\rho^*(L_{X_0}^\infty),\Lambda'})\Phi(P_{L_{X_0}})^*)-\frac{1}{2\pi i} {\rm tr\ log}
(\Phi(P_{L_{X_0}^\infty})\Phi(P_{L_{X_0}})^*)\\
&=\tau_\mu(P_{L_{X_0}}, P_{L_{X_0}^\infty}, P_{\rho^*(L_{X_0}^\infty),\Lambda'})
+\frac{1}{2\pi i} {\rm tr\ log}(\Phi(P_{\rho^*(L_{X_0}^\infty),\Lambda'})\Phi(P_{L_{X_0}^\infty})^*).
\end{split}
\end{equation}
Now, since in the decomposition $L^\infty_{X_0}= \tilde W  \oplus \tilde W' \oplus F_\nu^+$, $\overline\sigma(1/\beta)$ acts as the identity on the two latter summands we get using the same arguments as above and the fact that by the choice of logarithm as in (\ref{wind12}), one has for any $\alpha >0$:
\begin{equation}\label{wind123}
 \frac{1}{2\pi i}{\rm log}(e^{2\pi i\alpha})=-\frac{1}{2}+\{\alpha+\frac{1}{2}\}=\{\alpha\}'  \in (-\frac{1}{2},\frac{1}{2}],
\end{equation}
where here, $\{\cdot\}'$ means the modified fractional part introduced above Theorem \ref{theorem2}), the following:
\begin{equation}\label{trlog2}
\frac{1}{2\pi i} {\rm tr\ log}
(\Phi(P_{\rho^*(L_{X_0}^\infty),\Lambda'})\Phi(P_{L_{X_0}^\infty})^*) = \sum_{\alpha\in \Lambda'}(\{2({\rm deg}(z^\alpha(j)) +\sum^\mu_{i=1} w_i) + \frac{1}{2}\}-\frac{1}{2})
\end{equation}
which gives the assertion.
For the third part of the theorem one uses again the first part and the formula (\cite{nem2})
\[
\mathcal{V}(f)=\bigoplus_{\alpha \in \Lambda}\mathcal{W}_{exp(2\pi il(\alpha))}((-1)^{[l(\alpha)]+n})
\]
for the variation structure of a quasihomogeneous polynomial together with $l(\alpha)=\sum_{i=0}^n(\alpha_i+1)w_i$.\\
Now Corollary \ref{seidelfolg} is a direct consequence of the first part of the theorem for $\alpha=0$ and the result of Seidel in \cite{seidel} (see also \cite{klein3}).
\end{proof}
Note finally that Theorem \ref{theorem3} basically follows from the above proof by substituting the basis $\{\phi_1,\dots,\phi_\mu\}$ of $H^n(F,\mathbb{C})$ by the set of parallel sections $\{\hat \phi_1,\dots,\hat \phi_n\}$ of ${\bf H}^n(Z,\mathbb{C})$ s.t. $ev_F(\hat \phi_i)=\phi_i$ for all $i \in\{1,\dots,\mu\}$. Since the splitting (\ref{splitting}) is orthogonal which carries over  to $\mathcal{K}_{\tilde X}$ because of Corollary \ref{harmonicbla} (see Appendix A), a set of basis vectors analogous to (\ref{orthbasis}) will serve as an orthonormal basis for $\mathcal{K}_{\tilde X}$ inducing an orthonormal basis on $\tilde W$ using $\mathcal{L}$, the rest of the calculation is identical to the above.
\end{proof}

\section{Monodromy, boundary conditions and Reeb flow}
In this section, we will give a more direct interpretation of the action of $\overline \sigma(1/\beta)$ on the subspace $\hat W =V_{\tilde X} \oplus W_{\tilde X} \subset {\rm ker}\ A \oplus d(E_\nu^+)\oplus d^*(E_\nu^-)\subset L^2(\Omega^*_{\partial X})$ as defined in formula (\ref{limitcomposition}) resp. (\ref{strlagr3}) for the situation used in the formulation of Theorem \ref{theorem3}, thus giving a more 'geometric' interpretation of $L_{\rho_*(\tilde X^\infty)}:={\rm im}(P_{\rho_*(L_{\tilde X^\infty}),\Lambda})$. Recall that we assumed that $\tilde X$ equals the $\beta$-fold cyclic covering of the 'perturbed' Milnor fibration in the sense of Section 3 in \cite{klein2}, which has a canonical contact boundary trivialization, equipped with the submersion metric $\tilde g$, induced by the Euler vector field, as described in the discussion above Theorem \ref{theorem3}. \\
We first note that for $\Lambda=\Lambda' \subset \mathbb{Z}^{n+1}$, the evaluation
\[
\overline \sigma_\Lambda(1/\beta)=id\otimes \rho^*,
\]
on $H^*(\tilde X,\mathbb{C})\simeq H^*(S^1,\mathbb{C})\otimes H^*(F,\mathbb{C})$, where $\rho^*$ is the algebraic monodromy of the Milnor fibration $X$ is 'geometric' in the sense that:
\begin{lemma} There is a smooth isometry $\rho_{\tilde X}:\tilde X\rightarrow \tilde X$ so that $\rho_{\tilde X}$ covers the identity on $D^*_\delta$ and induces the map $id\otimes \rho^*$ on $H^*(\tilde X,\mathbb{C})$, where we understand to have chosen the identification $\tau:H^*(\tilde X,\mathbb{C})\simeq H^*(S^1,\mathbb{C})\otimes H^*(F,\mathbb{C})$ induced by the Euler vectorfield as in (\ref{splitting}), so one has
\begin{equation}\label{betacov45}
\begin{CD}
\tilde X @>>\rho_{\tilde X}> \tilde X \\
 @VV \tilde f V    @VV \tilde f V \\
\partial D^*_\delta  @>>id > \partial D_\delta^*,
\end{CD}
\end{equation}
and
\[
\tau\circ(\rho_{\tilde X})^*\circ \tau^{-1}=id\otimes \rho^*.
\]
\end{lemma}
\begin{proof}
Let $\tilde X_x=F, \ x \in S^1_\delta$ a fixed fibre and let $\rho_F=\sigma(1/\beta)$ be the representative of the geometric monodromy on $F$, considered as a fibre in $X$, induced by the weighted circle action $\sigma(t),t \in [0,1]$ on $X\subset \mathbb{C}^{n+1}$ (resp. its evaluation at $1/\beta$). Denoting $\tilde \sigma$ the lift of $\sigma$ to $\tilde X$, one defines for $u \in S^1_\delta$ setting $u=xe^{2\pi i t}$ and for any $z\in \tilde X$ s.t. $f(z)=u$:
\begin{equation}\label{bla567}
(\rho_{\tilde X})(z)=\tilde \sigma(t)\circ\rho_F\circ\tilde \sigma^{-1}(t)(z).
\end{equation}
Since $\tilde \sigma(t)$ is an isometry of $\tilde X$ with respect to its (lifted) induced metric, $\rho_X$ is an isometry
of $\tilde X$, as asserted. By definition we have:
\[
(\rho_{\tilde X})^*(z)=(\tilde \sigma^*)^{-1}(t)\circ(\rho_F)^*\circ(\tilde \sigma^*)(t)(z)
\]
on $T^+\tilde X$, recalling that the identification $\tau$ can be done by parallel transport along the flow of the Euler vector field, one infers that the action of $\rho_X$ on cohomology has exactly the asserted form.
\end{proof}
Let now 
\begin{equation}\label{boundaryaction}
\rho_{\tilde X}^*| := \rho_{\tilde X}^*|_{\partial \tilde X} \in \mathcal{B}(L^2(\Omega^*_{\partial \tilde X}))
\end{equation}
be the restriction of $\rho_{\tilde X}$ to $\partial X$ (well-defined by definition), acting on forms and let $\gamma:L^2(\Omega^*_{\partial \tilde X})\rightarrow L^2(\Omega^*_{\partial \tilde X})$ be as before. Note that $\rho_{\tilde X}^*|$, being an isometry with respect to the restricted metric on $\partial \tilde X$, induces an isometry on $L^2(\Omega^*_{\partial \tilde X})$ , i.e. this is the case for its restriction $(\rho_{\tilde X}^*|) |(F_0^-\oplus {\rm ker} \ A)\cap L_{\tilde X}$ which will be sufficient for $\rho_{\tilde X}^{\pm}(\hat W)$  ($\rho_{\tilde X}^{\pm}$ defined as in \ref{globalproj}) to define an isotropic subspace. 
So consider the projections onto the $\pm i$-eigenspaces of $\gamma$ given by $P_{\pm}=\frac{1}{\sqrt{2}}(I \mp i\gamma)$ and define the bounded linear operator
\begin{equation}\label{globalproj}
\rho_{\tilde X}^{\pm}:=P_+\circ (\rho_{\tilde X}^*|)^{-1}+P_-\circ \rho_{\tilde X}^* | \in \mathcal{B}(L^2(\Omega^*_{\partial \tilde X})).
\end{equation}
\begin{lemma}\label{lemmageom}
Considering the  notation in Theorem \ref{theorem3}, that is $L^\infty_{\tilde X}= \hat W  \oplus \hat W' \oplus F_\nu^+$ we have that $\rho_{\tilde X}^{\pm}(\tilde W)$ as given by (\ref{globalproj}) is a Lagrangian subspace of ${\rm ker}\ A \oplus d(E_\nu^+)\oplus d^*(E_\nu^-)\subset L^2(\Omega^*_{\partial X})$ and defining a Lagrangian
\begin{equation}\label{lag567}
L_{\rho^*|,\tilde X}^\infty:= \rho_{\tilde X}^{\pm}(\hat W)  \oplus \hat W' \oplus F_\nu^+,
\end{equation}
where $(\cdot)^*$ denotes the adjoint with respect to the $L^2$-inner product on $L^2(\Omega^*_{\partial \tilde X})$ or to put it differently, noting that $L_{\tilde X^\infty}\cap (F_0^-\oplus {\rm ker }\ A)=\hat W$ 
\[
L_{\rho^*|,\tilde X}^\infty= \rho_{\tilde X}^{\pm}\left(L_{\tilde X}^\infty\cap (F_0^-\oplus {\rm ker}\ A)\right)\oplus  \frac{L_{\tilde X}^\infty}{ L_{\tilde X}^\infty\cap (F_0^-\oplus {\rm ker }\ A)}.	
\]
we have (using the notation introduced in Theorem \ref{theorem3}):
\begin{equation}\label{etadiff341}
\tilde \eta(D_{P_{\rho^*(L_{\tilde X}^\infty)}})-\tilde \eta(D_{P_{L_{\tilde X}^\infty}})=\tilde \eta(D_{P_{L_{\rho^*|,\tilde X}^\infty}})-\tilde \eta(D_{P_{L_{\tilde X}^\infty}}),
\end{equation}
i.e. the latter difference equals the expression (2) in Theorem \ref{theorem3}.
\end{lemma}
\begin{proof}
Let, as before, $r:L^2(\Omega^*_{\tilde X})\rightarrow L^2(\Omega^*_{\partial \tilde X})$ be given by $r(\cdot)=i^*(\cdot)+i^*(*_{\tilde X}\cdot)$, where $i:\partial \tilde X \rightarrow \tilde X$ is the inclusion. Then, as we saw above, $r:\mathcal{K}_{\tilde X}\rightarrow \hat W=V_{\tilde X}\oplus  W_{\tilde X} \subset {\rm ker}\ A\oplus d(E_\nu^+)\oplus d^*(E_\nu^-)$ is an isomorphism. We now claim that
\begin{equation}\label{restriction3}
\rho_{\tilde X}^*|(\hat  W)=r\circ \rho_{\tilde X}^*\circ r^{-1}(\hat  W).
\end{equation}
But this follows since $\rho_{\tilde X}$ acts as an isometry on $\tilde X$ and thus commutes with $*$ acting on $\Omega^{\rm even}(\tilde X,\mathbb{C})$. Furthermore, since $\rho_{\tilde X}$ is an isometry it maps $\mathcal{K}_{\tilde X}$ to itself and realizes on the cohomological level, using
\begin{equation}\label{decompfibre}
\mathcal{K}_{\tilde X}=\sum_{\rm even}\tilde f^*\mathcal{H}^*(S^1,\mathbb{C}) \otimes \Gamma_0(\mathcal{K}_F),
\end{equation}
exactly the action $\overline \sigma_\lambda(1/\beta)$ on $\mathcal{K}_{\tilde X}$ as described in (\ref{harmonicaction}). Assuming that $i^*(\cdot)$ would preserve the $L^2$-inner products, we would be done, since then $\mathcal{L}=r$. We cannot expect this to hold, however since we know that $\rho_{\tilde X}^*|$ preserves the $L^2$ inner product and commutes with $\gamma$, we have setting $\hat W_\mp=(I\pm i\gamma)\hat W$ that $\hat W=\hat W_-\oplus \hat W_+$ and the following commutative diagram (note that since $\rho_{\tilde X}^*|$ preserves $\hat W$ and commutes with $\gamma$, it also preserves $\hat W_\pm$):
\begin{equation}\label{betacov451}
\begin{CD}
\hat W_- @>>\Phi(\hat W)> \hat W_+ \\
 @VV \rho_{\tilde X}^*| V    @VV  \rho_{\tilde X}^*| V \\
\hat W_-  @>>\Phi(\hat W)> \hat W_+.
\end{CD}
\end{equation}
Then the map associated to $\rho_{\tilde X}^{\pm}(\hat W)$ is 
\[
\Phi(\rho_{\tilde X}^{\pm}(\hat W_-))=\rho_{\tilde X}^{\pm}\circ \Phi(\hat W) \circ (\rho_{\tilde X}^{\pm})^{-1}:\hat W_-\rightarrow \hat W_+
\]
which is clearly unitary (being a composition of unitaries), so $\rho_{\tilde X}^{\pm}(\hat W)$ is an isotropic subspace of $L^2(\Omega^*_{\partial \tilde X})$. Now since we can use the restriction to $\partial \tilde X$ of an orthonormal basis of $\mathcal{K}_{\tilde X}$ which is associated to a set of basis elements $\phi_j\subset H^n(F,\mathbb{C})$ for an appropriate set of $n+1$-tuples $\{\alpha_j\} \subset \Lambda$ as described in the proof of Theorem \ref{theorem3} giving a basis $f_1,\dots,f_\mu$ for $\hat W$ and since on this basis, by (\ref{restriction3}), $\rho_{\tilde X}^*|$ acts as $r\circ\overline\sigma(1/\beta)\circ r^{-1}$, namely by multiplication of $U(1)$-elements, we deduce that $f_1,\dots,f_\mu$ are eigenvectors of the unitary mapping $\rho_{\tilde X}^*|$, hence can be chosen to be orthonormal. This proves (\ref{etadiff341}).
\end{proof}
{\it Remark.} Note that by comparing with (\ref{etadiff45}) the quantity $\tilde \eta(D_{P_{L_{\rho^*|,\tilde X}^\infty}})-\tilde \eta(D_{P_{L_{\tilde X}^\infty}})$ is determined by boundary data, namely the restriction of $\rho_{\tilde X}$ to $\partial \tilde X$, the 'tangential operator' $A$ of $\tilde X$, and an 'interior part' $r(\mathcal{K}_{\tilde X})\subset L^2(\Omega^*(\partial \tilde X))$. In fact, we can replace the left-hand side of (\ref{etadiff341}) by the quantity
\[
\tilde \eta(D_{P_{L_{\rho^*|,X_0}^\infty}})-\tilde \eta(D_{P_{L_{X_0}^\infty}})
\]
where $X_0=F\times S^1_\delta$ (this time, taking $\partial X_0$ as left boundary), $\rho_{X_0}=\sigma(1/\beta)\times id$ and $L_{\rho^*|,X_0^\infty}$ is defined analogous to (\ref{lag567}). To see this, note that the action of $\rho_{\tilde X}$ on $\mathcal{K}_{\tilde X}$ corresponds to the action of $\rho_{X_0}$ on $\mathcal{K}_{X_0}$ using the identification $\mathcal{K}_{X_0}\simeq \mathcal{K}_{\tilde X}$ as vectorspaces induced by (\ref{L2harmonicdecomp}), consequently we see that the restricted to the boundary actions on $\hat W$ resp. $W_0:=r(\mathcal{K}_{X_0})$ are equal under the induced identification $W_0\simeq \hat W$.\\
We finally observe the somewhat unexpected fact that $\rho_{\tilde X}^*|$ as defined in (\ref{boundaryaction}) and (\ref{bla567}) coincides with the action of the evaluation of the Reeb flow of a certain contact form on $L^2(\Omega^*(\partial \tilde X))$ at $1/\beta$ (identifying $S^1$ and $[0,1]/\{0,1\}$). Here we consider the Reeb flow as an action on $\partial \tilde X$ by its natural trivialization $\partial F\times S^1$ as considered in \cite{klein3} (Section 2.1, eq. (8)) (acting as the identity on the $S^1$-factor). Before discussing (possible) implications of this observation we first state it precisely:
\begin{folg}\label{reeb}
With the above notation, there is a contact structure, that is a family of contact forms $\eta \in \Omega^1(\partial \tilde X_x)$ on the boundary of the fibres $\partial \tilde X_x= \partial F$, so that the associated Reeb vector field $\mathcal{B}$ (i.e. $\eta(\mathcal{B})=1$), the contact distribution $\Theta\subset T\partial F$ (defined by $\eta(\Theta)=0$) and the gradient direction $N$ (w.r.t the restricted euclidean radius function) decompose $TF|{\partial F}$ as
\[
TF|\partial F=N \oplus \mathcal{B}\oplus \Theta,
\]
s.t. $N \perp \Theta$. Furthermore, the foliation given by the Reeb vector field has closed leaves, let $\sigma_\mathcal{B}$ be the associated flow $S^1\times \partial F\rightarrow \partial F$. Identifying $S^1$ and $[0,1]/\{0,1\}$ and understanding $\sigma_{\mathcal{B}}$ as a flow $\tilde \sigma_{\mathcal{B}}=\sigma_{\mathcal{B}} \times id_{S^1}$ on $\partial \tilde X=\partial F\times S^1$, we have using the trivialization given by (7) in (\cite{klein3}, Section 2.1)
\begin{equation}
\rho_{\tilde X}^*|=\tilde \sigma_{\mathcal{B}}(1/\beta))^* \in \mathcal{B}(L^2(\Omega^*_{\partial \tilde X})).
\end{equation}
\end{folg}
\begin{proof}
That there exists a contact form on $\partial F$ with the asserted properties follows from the considerations in Abe (\cite{abe}, Theorem 2).  Note that it is essential here that we used the 'perturbed' Milnor fibration as defined in Appendix A or Section 2.1 in \cite{klein3}. Our assertion then follows directly from the fact using (\ref{bla567}) for $z=(x,t)$ in $\partial X=\partial F\times S^1$ we have 
\[
\rho_{\tilde X}(z)=(\rho_F, {id})(z)=(\sigma(1/\beta),id)(z),
\]
where here, $\sigma$ denotes the weighted circle action, restricted to $\partial F$ and $id$ denotes the identity map on $S^1$. Now since the flow of $\mathcal{B}$ is by construction in Abe (\cite{abe}, Example 4) for the case of a weighted homogeneous polynomial given by the weighted circle action $\sigma$ (\ref{weightedfirst}), we arrive at the assertion.
\end{proof}
{\it Remark.} The Lemma suggests to replace the definition of the boundary condition under the $\overline \sigma(1/\beta)$ evaluation of the circle action on $L_{\tilde X}$, which uses the 'periods' of the sections $\Phi_1,\dots,\Phi_\mu$ associated to the $z^\alpha,\alpha \in \Lambda$, by an evaluation of a Reeb flow on $\hat W\subset L^2(\Omega^*(\partial \tilde X))$ or $W_0\subset L^2(\Omega^*(\partial \tilde X))$ without changing the value of the difference of eta-invariants in Theorem \ref{theorem3}. Since the Reeb foliation in question has closed curves as its leaves and the Reeb flow acts by isometries on $\partial \tilde X$, the associated circle action on $\partial \tilde X$ gives rise to a loop of Lagrangian subspaces in $(L^2(\Omega^*(\partial \tilde X)), \omega)$ ($\omega$ as introduced in Section 2.2 of \cite{klein2}) by letting it act on $\hat W$ resp. $W_0$ analogously to (\ref{globalproj}) and the question arises what would be its relation to the above discussed loops $P_{\Lambda,t}$ resp. their spectral flow as constructed in Theorem \ref{theorem3}.\\
To be more precise and to answer the above question for the case of the Reeb flow $\tilde \sigma_{\mathcal{B}}$ acting on $W_0$ we define a family of isotropic subspaces $W_0^t, t \in [0,1]$ by defining a family of bounded linear operators
\begin{equation}\label{globalproj5}
\tilde \sigma_{\mathcal{B},\pm}(t):=P_+\circ (\tilde \sigma_{\mathcal{B}}(t))^{-1}+P_-\circ \tilde \sigma_{\mathcal{B}}(t) \in \mathcal{B}(L^2(\Omega^*_{\partial \tilde X})), t \in [0,1],
\end{equation}
and setting $W_0^t=\tilde \sigma_{\mathcal{B},\pm}(t)(W_0)$. Then
\begin{equation}\label{reeblags}
L_{\tilde \sigma_{\mathcal{B}},\tilde X^\infty}(t):= W_0^t  \oplus \hat W' \oplus F_\nu^+,
\end{equation}
defines a family of Lagrangian subspaces in $(L^2(\Omega^*(\partial \tilde X)), \omega)$ (note that we could consider $\partial \tilde X$ as the boundary of $X$ without changing anything in the following arguments). We then have
\begin{theorem}\label{difforder}
Denote ${\rm SF}(\sigma_{\mathcal{B}})$ the spectral flow for the signature operator $D$ on $\tilde X$ associated to the family $P_t\in Gr(A),\ t \in [0,1]$ projecting onto the family (\ref{reeblags}). Then if for some $m \in \mathbb{N}$ we have $0=\rho^{m\beta}\in \pi_0({\rm Diff}(F,\partial F))$ for $\rho \in {\rm Diff}(F,\partial F)$ representing $\sigma(1/\beta)$ in $\pi_0({\rm Diff}(F))$ under the forgetful map, then $SF(\sigma_{\mathcal{B}})=0$.
\end{theorem}
{\it Remark.} As remarked in the introduction, the hypothesis of the Theorem is true by results of Krylov and Stevens (\cite{krylov}, \cite{stevens}), if $\partial F$ is a rational homology sphere. So in this case, ${\rm SF}(\sigma_{\mathcal{B}})$ is always zero.
\begin{proof}
Let $\tilde X^m$ the $k=m\cdot\beta$-fold cycling covering of $X$ equipped with the metric given by (\ref{metricbla}) (adopted to the case $m>1$ and 'deformed' near its boundary as described in Lemma 3.3 of \cite{klein2}), $\tilde f^m: \tilde X^m\rightarrow S^1$ the fibration analogous to the case $m=1$ in (\ref{betacov}) and let $0=\rho^{m\beta}\in \pi_0({\rm Diff}(F,\partial F))$. Then (compare Lemma 3.4 in \cite{klein3}) $\tilde X^m$ is smoothly cobordant to the trivial fibration $X_0=F\times S^1$ by a fibration $g:W\rightarrow S^1_\delta\times  [0,1]$ and the boundary of each slice $X_\tau=g^{-1}(S^1_\delta\times \{\tau\})$ is diffeomorphic to $S^1_\delta \times \partial F$. \\
{\it Claim}. We can equip the family $X_\tau$ with a smooth family of metrics so that $\partial X_\tau$ becomes isometric to $S^1_\delta \times \partial F$ with the product metric  for all $\tau \in [0,1]$ and coincides for $\tau=1$ with the metric given on $\tilde X^m$, for $\tau=0$ with the metric product $X_0=S^1_\delta\times F$.\\
To see that such a family of metrics exists note that following Lemma 2.15 in \cite{klein3}, the difference of symplectic parallel transport in $\tilde X^m$ along $t \mapsto \delta e^{2\pi i t}, t \in [0,1]$ and the flow of the lifted Euler vector field is given by a time dependent Hamiltonian flow $\Phi_H(t)(t), \ t \in [0,1]$ in the fibres, that is $\Phi_{\tilde X_f}(t)= \Phi_H(t)(t)\circ \Phi_{\Omega}(t)$ where for $t \in [0,1]$, $\Phi_{\Omega}(t):\tilde X^m_x\rightarrow \tilde X^m_{e^{2\pi it}x}$ denotes symplectic parallel transport and $\Phi_{\tilde X_f}:\tilde X^m_x\rightarrow \tilde X^m_{\Phi_{e^{2\pi it}x}}$ denotes the flow along the (lifted) Euler vector field $\tilde X_f$ on $\tilde X^m$ and so that $\Phi_H(1)$ equals the $k$-th power of the symplectic monodromy mapping $\rho \in {\rm Symp}(F, \partial F, \omega)$. Consider now a path $\rho^k_t, \ t\in [0,1]$ connecting $\rho^k$ to $Id$ in ${\rm Diff}(F,\partial F)$, that is $\rho^k_0=Id$, $\rho^k_1=\rho^k$. Let
\[
\tilde \Phi_\tau:[0,1]\times F \mapsto F, \tilde \Phi_\tau(t)= (\rho^k_{t(1-\tau)})^{-1}\circ \Phi_H(t)(\tau\cdot t), \ \tau \in [0,1].
\]
Then if $J \in End(TF)$ is the almost complex structure on $F=\tilde X^m_x$, for a fixed $x$, then for $\tau=1$, the given almost complex structure on $F \times \{t\},\ t \in [0,1]$ is $\Phi_H(t)_*(J)$. Correspondingly, we define for $\tau \in [0,1]$ an almost complex structure on $F \times \{t\}$ by $J^\tau_t=\tilde \Phi_\tau(t)_*(J)$. On the other hand, set for $\tau \in [-1,0]$ $\tilde \Phi_\tau(t)=\rho^k_{t(1-\tau)}$ and $J^\tau_t=\tilde \Phi_\tau(t)_*(J), \tau \in [-1,0], t \in [0,1]$ accordingly. Set $X_\tau= F \times_{\tilde \Phi_\tau}S^1, \tau \in [0,1]$, the family of mapping cylinders associated to $F$ with $\{\tilde \Phi_\tau\}_{\tau \in [-1,1]} \subset {\rm Diff}(F,\partial F)$ as above. Then  the pair $(\omega, F)$ and the family $J^\tau_t, \ \tau\in [-1,1], t \in [0,1]$ define a family of vertical metrics $g^v_\tau$ on $X_\tau$. Lifting the base metric $g_{S^1_\delta}$ from $S^1_\delta$ to a metric $\tilde g_{S^1_\delta}$ on the horizontal space $H_{\tilde \phi_\tau}$ on $X_\tau$ given by the vector field $d/dt (e^{2\pi it}x, \tilde \Phi_\tau(t)), \ t \in [0,1]$ on $X_\tau$, and setting
\[
g_\tau= g^v_\tau \oplus_{H_{\tilde \phi_\tau}} \tilde g_{S^1_\delta}, \  \tau\in [-1,1],
\]
we arrive at a family of Riemannian manifolds $(X_\tau, g_\tau), \tau \in [-1,1]$ so that $(X_1, g_1)$ is isometric to $\tilde X^m$ and $(X_{-1}, g_{0})$ is isometric to the metrically trivial fibration $F\times S^1$. Reparametrizing in $\tau$, we arrive at the assertion.\\
Now define for $\tau \in [0,1]$
\[
L_{X_\tau}\cap (F_0^-\oplus {\rm ker }\ A)=: W_\tau,
\] 
where $L_{X_\tau}$ is the image of the Calderon projector of $X_\tau$ and note that $W_0$ coincides with our prior definition. Then setting 
\[
W^t_{\tau}:=\left(P_+\circ ((\tilde \sigma^m(t))^*)^{-1}+P_-\circ \tilde \sigma^m(t)^*\right)(W_\tau), \ t, \tau \in [0,1],
\] 
where $\tilde \sigma^m:[0,1] \times \partial F\times S^1\rightarrow \partial F \times S^1$ equals $\tilde \sigma^m(t)(x,z)=(\sigma^m(t)x, e^{2\pi it}z)$, where $\sigma^m$ is the lift of the weighted circle action $\sigma$ to $\tilde X^m$, defines a family of Lagrangian loops
\[
L_{\tau, t}:= W^t_{\tau}  \oplus \hat W' \oplus F_\nu^+,
\]
so that $L_{0, t}$ coincides with the family (\ref{reeblags}) by the metric triviality of $X_0$. For $\tau \in [0,1]$ fixed, the family $L_{\tau, t}$ thus defines a spectral flow ${\rm SF}(\tau)$, so that ${\rm SF}(0)={\rm SF}(\sigma_{\mathcal{B}})$. On the other hand, since for $\tau=1$ and by Lemma \ref{harmonicbla} in Appendix A, $\tilde \sigma^m$ acts as the identity on $\tilde r^{-1}(W_1)=:\mathcal{K}_{\tilde X^m}$, where $\tilde r:\tilde X^m\rightarrow \partial \tilde X^m$ is equal to $r=i^*(*)+i^*$, $i:\partial \tilde X^m\hookrightarrow \tilde X^m$ being the inclusion, we get ${\rm SF}(1)=0$. The assertion now follows since we can always connect $W_0$ and $W_\tau$ for $\tau \in [0,1]$ by the smooth path $\tau \mapsto W_\tau$ in $\mathcal{L}_{Fred}$ and transform the family of loops $t \mapsto W^t_{\tau}$ into a family of {\it based} (at $W_0$) loops without changing the spectral flow by Theorem \ref{nicolaescu} and the additivity of the Maslov index under concatenation.
\end{proof}
Comparing with (\ref{seidelcondition}), we see that the non-vanishing of $SF(\sigma_{\mathcal{B}})\in \mathbb{Z}$ obstructs $\rho^\beta$ to have finite order in $\pi_0({\rm Diff}(F,\partial F))$ in the same way as the non-vanishing of $sf(\alpha=0)=\frac{1}{2}SF(\alpha=0)\in \mathbb{Z}$ as introduced in the Introduction resp. Theorems \ref{theorem2} and \ref{theorem3} obstructs the symplectic monodromy $\rho_s \in {\rm Symp}(F,\partial F,\omega)$ to be of finite order in $\pi_0({\rm Symp}(F,\partial F,\omega))$, provided $(\sum_{i=1}^\mu\beta_i-\beta)/\beta \notin \mathbb{Z}\setminus 0$. The latter fact of course is basically a consequence of Theorem 2.10 in \cite{klein3}. Summarizing, we have:
\begin{folg}\label{folgreeb}
Let $n\geq 2$. Assume $(\sum_{i=1}^\mu\beta_i-\beta)/\beta \notin \mathbb{Z}\setminus 0$. Then if $SF(\sigma_{\mathcal{B}})\neq sf(\alpha=0)$, then $\rho_s \in \pi_0({\rm Symp}(F,\partial F,\omega))$ is of infinite order. If in addition, $\partial F$ is a rational homology sphere, then the map $\pi_0({\rm Symp}(F,\partial F,\omega))\rightarrow \pi_0({\rm Diff}(F,\partial F))$ has an infinite kernel.
\end{folg}
\begin{proof}
Assume $SF(\sigma_{\mathcal{B}})\neq sf(\alpha=0)$ and $\rho_s \in \pi_0({\rm Symp}(F,\partial F,\omega))$ would be of finite order, then $\rho_s$ is of finite order in 
$\pi_0({\rm Diff}(F,\partial F))$ and by Lemma \ref{difforder}, $SF(\sigma_{\mathcal{B}})=0$. But then, $sf(\alpha=0)\neq 0$ and by the fact that $sf(\alpha=0)=\sum_{i}\beta_i-\beta$ and Theorem 2.10 in \cite{klein3}, $\rho_s$ is of infinite order in $\pi_0({\rm Symp}(F,\partial F,\omega))$.
\end{proof}
Note that by Seidel's result (\cite{seidel}) resp. (\ref{seidelcondition}), the condition $(\sum_{i=1}^\mu(\beta_i-\beta)/\beta \notin \mathbb{Z} \setminus 0$ is not essential. Furthermore, by the arguments in the respective introductions of \cite{klein3} and this article, we conjecture that the above Corollary remains true for the 'full spectral flow' ${\rm SF}(D_{P_{\Lambda}(t)})_{t \in [-1,2]}$ (compare Theorems \ref{theorem2} and \ref{theorem3}), that is the inequality of the latter and $SF(\sigma_{\mathcal{B}})$ would already give a sufficient condition for $\rho_s \in \pi_0({\rm Symp}(F,\partial F,\omega))$ to be of infinite order (compare also Section 4.2 of \cite{klein1}).

\section{Appendix A}\label{app2}
In this Appendix we will gather briefly some well-known results on the splitting of the exterior derivative on fibrations carrying Riemannian submersion metrics, these results are needed in Section \ref{monodr}, and draw some immediate conclusions. The arguments concerning Proposition \ref{d-decomp} are drawn from \cite{mazzeo}, further references are for instance \cite{getzler} and \cite{bilo}.\\
Let $G = \pi^*(h) + k= (\cdot, \cdot)$ be a metric orthogonalizing the splitting $TY=T^H Y\oplus T^VY$
of the tangent bundle of the total space of a fibration $Y$, where $\pi:Y \to B$ and $\phi^{-1}(b) = F_b$, denoting $P^V: TY \to T^V Y$, $P^H: TY \to T^H Y$ the orthogonal projections. Identifying tangent bundle $TB$ naturally via $\pi_*$ with $T^H Y$, we denote the lift of a section $X \in \mathcal{C}^\infty(B;TB)$ by $\tilde{X}$. In the following, we will denote sections of $T^V Y$ and $T^H Y$ by $V_1, V_2, \ldots$, and $\tilde{X}_1, \tilde{X}_2, \ldots$, respectively, $\nabla^L$ denotes the Levi-Civita connection of $G$.\\
Consider now the curvature of the horizontal distribution, resp. the second fundamental form of the vertical fibration. The latter is the bilinear form on $T^VY$ defined by
\begin{equation}
\Pi_{\tilde{X}}(V_1,V_2) =
\left( \nabla_{V_1}^L V_2, \tilde{X} \right).
\end{equation}
Now let $\Pi(V_1,V_2)$ be the horizontal vector given by
$$
\left( \Pi(V_1,V_2), \tilde{X} \right ) = \Pi_{\tilde{X}}(V_1,V_2),
$$
and let $\Pi_{\tilde{X}}(V_1)$ denote the vertical vector
determined by 
$$
\left( \Pi_{\tilde{X}}(V_1),V_2 \right) = \Pi_{\tilde{X}}(V_1,V_2).
$$
On the other hand define
\begin{equation}
\mathcal{R}(\tilde{X}_1,\tilde{X}_2) = P^V([\tilde{X}_1,\tilde{X}_2]),
\end{equation}
resp. the horizontal vector
$\hat{\mathcal{R}}_V(\tilde{X}_1)$ by
\begin{equation}
\left( \hat{\mathcal{R}}_V(\tilde{X}_1), \tilde{X}_2 \right)
= \left( \mathcal{R}(\tilde{X}_1,\tilde{X}_2), V\right)
= \left( [\tilde{X}_1,\tilde{X}_2],V\right).
\end{equation}
Now one defines the following connection which preserves the splitting of $TY$,
\[
\nabla := \left(P^V \nabla^L\right) \oplus \nabla^B,
\]
or more explicitly,
\[
\begin{array}{rcl}
\nabla_{V_1} V_2 & = & P^V (\nabla^L_{V_1}V_2), \\
\nabla_V \tilde{X} & = & 0,
\end{array}
\qquad
\begin{array}{rcl}
\nabla_{\tilde{X}}V &= &P^V(\nabla^L_{\tilde{X}}V) =
[\tilde{X},V] - \Pi_{\tilde{X}}(V), \\
\nabla_{\tilde{X}_1}\tilde{X}_2 & = & \widetilde{(\nabla^B_{X_1}X_2)}.
\end{array}
\]
We want to express the de Rham differential $d_Y$ and its adjoint in terms of $\nabla$, $\Pi$ and $\mathcal{R}$. Being metric connections, $\nabla^L$ and $\nabla$ act on 1-forms, hence extended as derivations, on $n$-forms, by duality.
Let $\{e_i\}$, $i=1, \ldots, f$ and $\{f_\mu\}$, $\mu = 1, \ldots, b$ be
orthonormal frame fields for $F$ and $B$, respectively, and
$\{e^i\}$, $\{f^\mu\}$ the dual coframe fields. Then
\begin{equation}
d_Y = \sum_{i=1}^f e^i \wedge \nabla^L_{e_i} + \sum_{\mu=1}^b
f^\mu \wedge \nabla^L_{f_\mu},
\end{equation}
analogously for $d_F$ and $d_B$. Using Proposition 12 in \cite{mazzeo} and comparing with the formulas for $\nabla$, one gets
\begin{equation}
d_Y e^j = d_F e^j + \sum f^\mu \wedge \nabla_{f_\mu}e^j
- \sum \left( \left( \Pi_{f_\mu}(e_i),e_j\right)
f^\mu \wedge e^i + \frac{1}{2}  \left(
\mathcal{R}(f_\mu,f_\nu),e_j\right)
f^\mu \wedge f^\nu\right),
\label{blaexp}
\end{equation}
and 
\begin{equation}
d_Y f^\mu = d_B f^\mu.
\end{equation}
Extend this to forms of higher degrees and note that the splitting of $TY$ induces a decomposition
 \[
\Omega^*(Y)=\Omega^{p,q}(Y)=\Omega^p(B)\ \hat\otimes \ \Omega^q(Y,T^VY)
\]
which  is preserved by $\nabla$. Thus for $\omega \in \Omega^{p,q}(Y)$, with $\omega = \pi^*(\alpha)\wedge \beta$, $\alpha \in \Omega^p(B)$ and $\beta \in \mathcal{C}^\infty(Y,\Lambda^q((T^VY)^*)$ one defines
\[
d_F \left(\pi^*(\alpha) \wedge \beta\right) = (-1)^p\pi^*(\alpha ) \wedge d_F \beta
\]
and
\[
\tilde{d}_B \pi^*(\alpha) \wedge \beta =
\pi^*(d_B \alpha) \wedge \beta + (-1)^p\pi^*(\alpha)\wedge
\left(\sum_{\mu}f^\mu \wedge \nabla_{f_\mu} \beta \right).
\]
Then (\ref{blaexp}) writes as
\[
d_Y e^j = d_F e^j + \tilde{d}_B e^j - \Pi(e^j) -\frac{1}{2} \mathcal{R}(e^j),
\]
where
\[
\Pi(e^j) =  \Pi_{\mu i j}\, f^\mu \wedge e^i, \qquad
\mathcal{R}(e^j) = \mathcal{R}_{\mu\nu j}f^\mu \wedge f^\nu.
\]
To simplify notation, let ${\rm R} = -\frac{1}{2} \mathcal{R}$. Then one has the final result:
\begin{prop} [\cite{mazzeo}] \label{d-decomp}
$d_Y = d_F + \tilde{d}_B - \Pi + {\rm R}$,  $\delta_Y = \delta_F + (\tilde{d}_B)^* - \Pi^* + {\rm R}^*$.
\end{prop}
\begin{folg}
With the notation from Chapter \ref{monodr}, the cohomology of the total space of the bundle $\tilde X$ (see \ref{betacov}), equipped with the submersion metric $\tilde g$, can be written as
\begin{equation}\label{eq523}
H^*(\tilde X,\mathbb{C})\simeq\tilde f^*H^*(S^1,\mathbb{C})\otimes \Gamma_0({\bf H}^*(\tilde Z,\mathbb{C})),
\end{equation}
where ${\bf H}^*(\tilde Z,\mathbb{C})$ denotes the $\mathbb{Z}$-graded vector bundle whose fibre over $u \in S^1$ is the cohomology of the complex $\Omega^*(\tilde Z_u,\mathbb{C})$ (note $\tilde Z_u=F$). $\Gamma_0$ indicates global parallel sections $s$ over $S^1$ that satisfy $\nabla_{\tilde X}s=L_{\tilde X}s=0$.
\end{folg}
\begin{proof}
The proof is immediate since with respect to $\tilde g$, the fibres of $\tilde X$ are totally geodesic, hence the second fundamental form vanishes. Furthermore, since the base is $S^1$, the horizontal curvature term also vanishes.
\end{proof}
For the following note that the {\it extended $L^2$-sections} of $\Omega^*(\tilde X_\infty,\mathbb{C})$ were defined in Atiyah \cite{Atiyah} to be sections $u$ which are locally $L^2$ and sucht that, with respect to coordinates $(y,t)$ on $\partial \tilde X\times (-\infty,0]$ over the collar one has 
\[
u(y,t)=g(y,t)+f_\infty(y),
\]
where $g$ is $L^2$ and $f_\infty\in {\rm ker}\ A$. As it well-known (cf. \cite{lesch}), for the signature operator $D$ on $\tilde X$ the extended $L^2$-harmonic forms on $\tilde X_\infty$ (extended $L^2$-sections satisfying $d\beta=0$and $d*\beta=0$ on $\tilde X_\infty$) are isomorphic to the solutions $\beta$ of $D\beta=0$ on $\tilde X$ satisfying the condition $r(\beta) \in F_0^-\oplus {\ker A}$, the latter we denoted in Chapter \ref{monodr} as $\mathcal{K}_{\tilde X}$. The following Lemma shows how $\mathcal{K}_{\tilde X}$ decomposes using the fact that $\tilde X$  has totally geodesic fibres using the (lift of the) submersion metric introduced in Section 3.1 of \cite{klein2} for $X$.
\begin{folg}\label{harmonicbla}
Let $\mathcal{K}_{\tilde X}$ be as in Chapter \ref{monodr} the space of extended $L^2$-solutions of the signature operator on $\tilde X_\infty$, then one has ($\sum_{\rm even}$ indicating summation over even forms)
\begin{equation}\label{L2harmonicdecomp}
\mathcal{K}_{\tilde X}=\sum_{\rm even}\tilde f^*\mathcal{H}^*(S^1,\mathbb{C}) \otimes \Gamma_0(\mathcal{K}_F),
\end{equation}
where $\mathcal{K}_F$ denotes representatives of ${\bf H}^{*}(\tilde Z,\mathbb{C})$ being for any $u\in S^1$ fibrewise extended $L^2$-harmonic forms on $\tilde Z_{u,\infty}$, where the latter denotes the elongation $\tilde Z_u$ along its metric collar (see \cite{Atiyah}). In particular, these representatives satisfy $d_F\beta=\delta_F\beta=0$ and decompose for any $u \in S^1$ as (set $F=\tilde Z_u$)
\begin{equation}\label{bla543}
\mathcal{K}_F|{Z_u} \simeq \left(im(H^{*}(F,\partial F,\mathbb{C})\rightarrow H^{*}(F,\mathbb{C}))\oplus im(H^*(F,\mathbb{C})\rightarrow H^*(\partial F, \mathbb{C}))\right),
\end{equation}
$\mathcal{H}^*(S^1,\mathbb{C})$ denotes the closed and coclosed forms on $S^1$.
\end{folg}
\begin{proof}
Using the splitting formulae for $d_{\tilde X}$ and $\delta_{\tilde X}$ from above and the fact that $\mathcal{K}_{\tilde X}$ consists of $\beta \in \Omega^*(\tilde X,\mathbb{C})$ so that $r(\beta) \in F_0^-\oplus {\rm ker}\ A$ which is equivalent (see Lesch/Kirk (\cite{lesch} and Atiyah et al.\cite{Atiyah})) to the conditions  $d\beta=d(*\beta)=0$ and $\beta$ being an extended $L^2$-section of $\Omega*(\tilde X,\mathbb{C})$, we see since $\tilde f^*\mathcal{H}^*(S^1,\mathbb{C}) \otimes\Gamma_0(\mathcal{K}_F)$ describes sections in $\Omega^*(\tilde X,\mathbb{C})$ being closed and coclosed and are extended $L^2$-solutions of $D\beta=0$ on $\tilde X_\infty$, that the dimensions of the former and the latter coincide, hence the two spaces are equal, which implies  (\ref{L2harmonicdecomp}). Equation (\ref{bla543}) is shown in Atiyah et al. (\cite{Atiyah}).
\end{proof}

\end{document}